\documentclass[12pt,dvipsnames]{amsart}

\usepackage[dvipsnames]{xcolor}
\usepackage{amsmath,amsthm,amssymb}
\usepackage{enumitem}
\usepackage{tikz}
\usepackage{hyperref}
\usepackage{amsfonts,marvosym,amscd}
\usepackage{mathtools}
\usepackage{mathscinet}
\usepackage{wasysym} 
\usepackage{graphicx}
\usepackage{psfrag}
\usepackage{tikz-cd} 
\usepackage{geometry}
\usepackage{adjustbox}
\usepackage{amssymb}
\usepackage[normalem]{ulem}
\usepackage{float}
\usepackage{tabstackengine}
\setstackgap{L}{.75\baselineskip}
\usetikzlibrary{decorations.pathmorphing}
\usetikzlibrary{positioning}
\usetikzlibrary{calc}

\newtheorem{theorem}{Theorem}[section]
\newtheorem{lemma}[theorem]{Lemma}
\newtheorem{proposition}[theorem]{Proposition}
\newtheorem{corollary}[theorem]{Corollary}

\theoremstyle{definition}
\newtheorem{definition}[theorem]{Definition}
\newtheorem{notation}[theorem]{Notation}
\newtheorem{example}[theorem]{Example}

\theoremstyle{remark}
\newtheorem{remark}[theorem]{Remark}

\makeatletter
\newcommand\@dotsep{4.5}
\def\@tocline#1#2#3#4#5#6#7{\relax
  \ifnum #1>\c@tocdepth 
  \else
    \par \addpenalty\@secpenalty\addvspace{#2}%
    \begingroup \hyphenpenalty\@M
    \@ifempty{#4}{%
      \@tempdima\csname r@tocindent\number#1\endcsname\relax
    }{%
      \@tempdima#4\relax
    }%
    \parindent\z@ \leftskip#3\relax \advance\leftskip\@tempdima\relax
    \rightskip\@pnumwidth plus1em \parfillskip-\@pnumwidth
    #5\leavevmode\hskip-\@tempdima{#6}\nobreak
    \leaders\hbox{$\m@th\mkern \@dotsep mu\hbox{.}\mkern \@dotsep mu$}\hfill
    \nobreak
    \hbox to\@pnumwidth{\@tocpagenum{\ifnum#1=1\fi#7}}\par
    \nobreak
    \endgroup
  \fi}
\AtBeginDocument{%
\expandafter\renewcommand\csname r@tocindent0\endcsname{0pt}
}
\def\l@subsection{\@tocline{2}{0pt}{2.5pc}{5pc}{}}
\makeatother

\newcommand{\D}{\mathcal{D}}

\newcommand{\ba}{\mathbf{a}}
\newcommand{\bv}{\mathbf{v}}
\newcommand{\bw}{\mathbf{w}}
\newcommand{\bx}{\mathbf{x}}
\newcommand{\by}{\mathbf{y}}
\newcommand{\bz}{\mathbf{z}}
\newcommand{\bg}{\mathbf{g}}

\newcommand{\Ds}{\mathcal{D}_{\supp}}

\newcommand{\dual}[1]{{#1}^{\ast}}

\newcommand{\tcs}[1]{\tabbedCenterstack{#1}}
\newcommand{\pspan}[1]{\mathrm{Span}_{\geqslant 0}\left\{#1\right\}} 
\newcommand{\lspan}[1]{\mathrm{Span}\left\{#1\right\}} 
\newcommand{\ip}[1]{\langle #1 \rangle} 
\newcommand{\sbm}[1]{{\let\amp=&\left(\begin{smallmatrix}#1\end{smallmatrix}\right)}}
\newcommand{\thin}[1]{T(#1)} 

\newcommand{\homa}{\mathrm{Hom}_{A}}
\DeclareMathOperator{\mc}{mod}
\DeclareMathOperator{\supp}{supp}
\DeclareMathOperator{\codim}{codim}
\DeclareMathOperator{\dimu}{\underline{dim}}
\DeclareMathOperator{\tp}{top}
\DeclareMathOperator{\soc}{soc}
\DeclareMathOperator{\brick}{brick}
\DeclareMathOperator{\simp}{simp}

\newcommand{\g}[1]{\bg^{#1}} 

\renewcommand{\emptyset}{\varnothing}

\renewcommand{\leq}{\leqslant} 
\renewcommand{\geq}{\geqslant} 

\newcommand{\sib}[1]
{{\textcolor{orange}{ #1}}}
\newcommand{\yad}[1]
{{\textcolor{violet}{#1}}}
\newcommand{\hip}[1]
{\textcolor{blue}{ #1}}
\newcommand{\nick}[1]
{\textcolor{red}{ #1}}

\newcommand{\aran}[1]
{\textcolor{black!30!green}{ #1}}

\title{Stability spaces of string and band modules}

\author[Schroll]{Sibylle Schroll}
\address{Abteilung Mathematik, Department Mathematik/Informatik der Universit\"at zu K\"oln, Weyertal 86-90, 50931 Cologne, Germany}
\email{schroll@math.uni-koeln.de}

\author[Tattar]{Aran Tattar}
\address{Abteilung Mathematik, Department Mathematik/Informatik der Universit\"at zu K\"oln, Weyertal 86-90, 50931 Cologne, Germany}
\email{atattar@uni-koeln.de}

\author[Treffinger]{Hipolito Treffinger}
\address{Universi\'e de Paris, B\^atiment Sophie Germain, 5 rue Thomas Mann, 75205 Paris Cedex 13, France}
\email{treffinger@imj-prg.fr}

\author[Valdivieso]{Yadira Valdivieso}
\address{Department of Mathematics, Universidad de las Am\'ericas Puebla, Mexico}
\email{yadira.valdivieso@udlap.mx}

\author[Williams]{Nicholas J. Williams}
\address{Graduate School of Mathematical Sciences, The University of Tokyo, 3-8-1 Komaba, Meguro-ku, Tokyo 153-8914, Japan}
\email{williams@ms.u-tokyo.ac.jp}

\begin{document}

\begin{abstract}
The stability space of a module is the cone of vectors which make the module semistable. These cones are defined in terms of inequalities; in this paper we draw insights from considering the dual description in terms of non-negative linear spans. We show how stability spaces of thin modules are related to order polytopes. In the case of non-thin modules, we show how the stability spaces of string and band modules are related to the stability spaces of the thin modules corresponding to the abstract string and band. We use this to analyse the way in which the stability space of a band module is the limit of stability spaces of string modules. Namely, the stability space of the band module is a union of cones, each of which is the limit of the stability spaces of a family of string modules.
\end{abstract}


\maketitle

\tableofcontents

\section{Introduction}

Stability conditions are a classical concept in algebraic geometry used to construct moduli varieties \cite{mumford}. In the 90s they were introduced for representations of  finite dimensional algebras \cite{Scho91,King94} and, more generally, for abelian categories \cite{Ru97}.

For a finite-dimensional algebra $A$, the stability conditions over $\mc A$ induce a partition of $\mathbb{R}^{n}$ \cite{King94}, known as the \emph{wall-and-chamber} structure.
The wall-and-chamber structure forms one component of a scattering diagram associated to the finite-dimensional algebra in \cite{BridgelandScat}---the other being a group whose elements correspond to the walls, which induce an action known as \emph{wall-crossing}.
Scattering diagrams were originally introduced in algebraic geometry \cite{GrossSiebert, KontsevichSoibelman} to study wall-crossing phenomena in Donaldson--Thomas theory.
They were subsequently applied with remarkable effect in \cite{GHKK} to solve several open problems in the theory of cluster algebras.

The theory of cluster algebras itself has had a sizeable impact on the representation theory of algebras via categorifications of cluster algebras, using such concepts as cluster categories \cite{CCS, BMRRT, Amiot} and quivers with potential \cite{DWZ}.
These categorifications shed light on previously-undescribed homological properties of categories of modules over algebras, and led to the introduction of $\tau$-tilting theory \cite{AIR14}.
In turn, it has been shown that the wall-and-chamber structure of an algebra is governed by the $\tau$-tilting theory of the algebra \cite{BST19, Asai}.

In this paper we are interested only in the wall-and-chamber structures of algebras, rather than the full scattering diagrams---meaning that we ignore the group elements associated to walls. The wall-and-chamber structure of an algebra $A$ is built using the notion of stability conditions introduced in \cite{King94} as follows (see Section~\ref{sec:background} for more details).
For each $A$-module $M$ the \textit{stability space} $\D(M)$ is the set of vectors $\bv$ in $\mathbb{R}^n$ making the module $M$ $\bv$-semistable. 
The union of all of the stability spaces of non-zero objects in $\mc A$ forms the wall-and-chamber structure of $A$. Here the \emph{walls} are the stability spaces of codimension one, while the \emph{chambers} correspond to the open connected subsets of the complement of the closure of the walls.

It follows from the definition of stability conditions that $\D(M)$ is the intersection of finitely many half-spaces; in other words, $\D(M)$ is a \emph{polyhedral cone}. 
The Minkowski--Weyl Theorem, a central result in Linear Programming, states that any cone which is a finite intersection of half-spaces is a non-negative linear span of a finite set of vectors. 
We investigate stability spaces $\D(M)$ from this dual point of view. We begin by considering the case where $M$ is a thin module, that is, where the composition factors of $M$ are pairwise non-isomorphic. We obtain the following description of the stability space of a thin module using the description of the facial structure of a stability space from \cite{Asai}.

\begin{proposition}[Proposition~\ref{prop:stability space of thin modules}]\label{prop:intro:thin}
Let $M$ be a thin sincere representation of an acyclic quiver $Q$. 
A minimal generating set of $\D(M)$ is given by \[\pspan{e_{i} - e_{j} \mid i \to j \mbox{ is an arrow of }Q},\] where $\pspan{-}$ denotes the non-negative linear span.
\end{proposition}

Proposition~\ref{prop:intro:thin} has interesting connections with other areas of mathematics.
An acyclic quiver $Q$ defines a poset $P$. If $M_{P}$ is a thin sincere representation $Q$, then we show also that the faces of the stability space of $M_{P}$ correspond to connected compatible partitions of $P$. This recovers a result of Geissinger and Stanley describing the faces of the order polytope of a poset \cite{Geissinger81,Stanley86}.
We also show that Proposition~\ref{prop:intro:thin} has a natural microeconomic interpretation when the poset defined by the quiver is interpreted as a set of preferences between certain goods.

We then show how Proposition~\ref{prop:intro:thin} may be used to describe the stability spaces of string and band modules. Here, given a string or a band, we may consider the thin representation of the quiver with the same shape as the string or the band. We obtain the following description of the stability space of the string or band module.

\begin{theorem}[Theorems ~\ref{thm:big-small} and \ref{thm:big-small band}]
Let $M$ be a string $A$-module or a band $A$-module over a special biserial $K$-algebra $A$ and let $T$ be a thin representation corresponding to the string or band.
Then the stability space $\mathcal{D}(M)$ is the intersection of $\mathcal{D}(T)$ with the hyperplanes which set the coordinates of repeated vertices in the string or band equal to each other.
\end{theorem}

We then apply this theorem to describe in detail how the stability spaces of band modules are the limits of stability spaces of string modules. For a more formal statement of this theorem, see Theorem~\ref{thm:limit_nonthin}.

\begin{theorem}[Theorem~\ref{thm:limit_nonthin}]\label{thm:intro:convergence}
Let $A$ be a special biserial algebra and let $b$ be a band with $\D(M(b, 1, \lambda))$ the stability space of the corresponding band module.
Then $\D(M(b, 1, \lambda))$ is a union of cones such that each cone is the limit of a particular family of string modules.
\end{theorem}

What is interesting is that, in general, one cannot approximate the whole stability space of the band module by a single family of string modules. Instead, it is necessary to decompose the stability space of the band module into smaller cones in this way. We illustrate this in Example~\ref{ex:convergence}.

Theorem~\ref{thm:intro:convergence} is related to recent work on $g$-tameness \cite{AokiYurikusa,KPY20,Asai22}. There is a subfan of the wall-and-chamber structure known as the `$g$-vector fan'. Algebras such that this fan is dense were called `$g$-tame' in \cite{AokiYurikusa}, where it was shown that complete special biserial algebras were $g$-tame. This implies that every point in $\mathbb{R}^{n}$ is a limit of cones in the $g$-vector fan.

This paper is structured as follows.
In Section~\ref{sec:background}, we give the necessary background to the paper on stability spaces and string and band modules.
In Section~\ref{sec:thin}, we describe the stability spaces of thin modules and show how this description relates to order polytopes and may be interpreted in terms of microeconomics.
In Section~\ref{sec:modules}, we show how the stability spaces of string and band modules may be described.
Finally, in Section~\ref{sec:w&cspecial}, we apply the results of the preceding section to describe how stability spaces of band modules are the limits of stability spaces of string modules.

\subsection*{Acknowledgements} This  paper is the product of a working group at the Department of Pure Mathematics at the University Leicester where the project originated. At the time of writing, it does not seem that there will be many more ideas born in the same place. We express deep support for the Pure Mathematics researchers at the University of Leicester.

SS and HT were supported by the EPSRC through the Early Career Fellowship, EP/P016294/1.
HT is also supported by the European Union’s Horizon 2020 research and innovation programme under the Marie Sklodowska-Curie grant agreement No 893654. 
YV is also supported by the European Union’s Horizon 2020 research and innovation programme under the Marie Sklodowska-Curie grant agreement No H2020-MSCA-IF-2018-838316.
SS acknowledges partial support from the DFG through  the project SFB/TRR 191 Symplectic Structures in Geometry, Algebra and Dynamics (Projektnummer 281071066– TRR 191).
SS, HT, YV, and NJW would like to thank the Isaac Newton Institute for Mathematical Sciences, Cambridge, for support and hospitality during the programme “Cluster Algebras and Representation Theory” where part of the work on this paper was undertaken. This work was supported by EPSRC grant no EP/K032208/1 and the Simons Foundation.
NJW is supported by a JSPS International Short-term Post-doctoral Research Fellowship.

\section{Background}\label{sec:background}

Throughout this paper we let $A$ be a finite-dimensional algebra over a field $K$ of the form $A = KQ/I$ where $Q$ is a quiver and $I$ is an admissible ideal of the path algebra $KQ$. From Section~\ref{sec:modules}, we will assume that $K$ is algebraically closed. Let $Q_{0}$ be the set of vertices of $Q$ and let $Q_{1}$ be the set of arrows. If $\alpha$ is an arrow of $Q$, then we denote the source of $\alpha$ by $s(\alpha) \in Q_{0}$ and the target of $\alpha$ by $t(\alpha) \in Q_{0}$. For $\alpha$ and $\beta$ arrows in $Q$ with $t(\alpha) = s(\beta)$, we write $\alpha\beta$ for the non-zero product of $\alpha$ and $\beta$ in $KQ$.

An important family of algebras that we will study in this article is the family of special biserial algebras \cite{BR87, WW85}. 
These are algebras $A=KQ/I$ which are defined by combinatorial properties of their quiver $Q$ and ideal $I$.
At the level of the quiver, an algebra $A=KQ/I$ is said to be \emph{special biserial} if for every vertex $v \in Q_0$ there are at most two arrows $\alpha, \beta \in Q_1$ such that $t(\alpha)=t(\beta)=v$ and at most two arrows $\gamma, \delta \in Q_1$ such that $s(\gamma)=s(\delta)=v$.
At the level of relations, if there are arrows $\alpha, \gamma, \delta$ such that $t(\alpha)=s(\gamma)=s(\delta)$, then $\alpha\gamma \in I$ or $\alpha\delta \in I$ or both.
Dually, if there are arrows $\alpha, \beta, \gamma$ such that $t(\alpha)=t(\beta)=s(\gamma)$, then $\alpha\gamma \in I$ or $\beta\gamma \in I$ or both.

\subsection{Representations of quivers and stability spaces}

We denote by $\mc A$ the category of finite-dimensional right $A$-modules. 
It is well-known that this category is equivalent to the category of \emph{representations} $(M_{i}, \phi_{\alpha})_{i \in Q_{0}, \alpha \in Q_{1}}$ of the bound quiver $(Q, I)$. 
Here a \emph{representation} of $Q$ consists of a finite dimensional vector space $M_{i}$ for every $i \in Q_{0}$ and a linear map $\phi_{\alpha}\colon M_{i} \to M_{j}$ for every arrow $\alpha\colon i \to j$ in $Q_{1}$. 
Given a path $w = \alpha_{1}\dots \alpha_{n}$ in $Q$, we write $\phi_{w} = \phi_{\alpha_{n}} \dots \phi_{\alpha_{1}}$. Similarly, given a linear combination of paths $\rho = \sum_{i = 1}^{r} \lambda_{i}w_{i}$, we write $\phi_{\rho} = \sum_{i = 1}^{r}\lambda_{i}\phi_{w_{i}}$. Then $(M_{i}, \phi_{\alpha})_{i \in Q_{0}, \alpha \in Q_{1}}$ is a \emph{representation of the bound quiver} $(Q, I)$ if $\phi_{\rho} = 0$ for all $\rho \in I$.

The \emph{dimension vector} $\dimu M$ of a representation $Q$ of $M$ is the vector with entries $\dim M_{i}$. We denote by $K_{0}(A)$ the Grothendieck group of $\mc A$. This is a free abelian group of rank $n$, where $n$ is the number of vertices of the quiver. Given a $A$-module $M$, we write $[M]$ for the class of $M$ in $K_{0}(A)$. We have that $[M] = \dimu M$ as vectors of integers.

We say that an $A$-module $M$ is \emph{thin}  if
all composition factors of $M$ are pairwise non-isomorphic. 
As consequence, if a module $M$ is thin, we have that the dimension vector $\dimu M$ has  only  0 and 1 as entries.

Given an $A$-module $M$, we define the \emph{support} $\supp M$ of $M$ to be the full subquiver of $Q$ with vertices \[\supp_{0}M := \{i \in Q_{0} \mid M_{i} \neq 0\}.\]
We say that $M$ is \emph{sincere} if $\supp_{0}M = Q_{0}$.

Let $M$ be an $A$-module. 
Let\[P_{-1}\longrightarrow P_0\longrightarrow M\longrightarrow 0\] be the minimal projective presentation of $M$, where $P_0=\bigoplus_{i=1}^n P(i)^{a_i}$ and $P_{-1} = \bigoplus_{i=1}^n P(i)^{b_i}$. 
Then the \emph{$g$-vector} of $M$ is defined to be \[\g{M} = (a_1-b_1, a_2-b_2,\dots, a_n-b_n).\]

\subsection{String and band modules}

We now recall the definitions of string and band modules, for which we need the following combinatorics. 
We refer to \cite{BR87} for more details on these modules.
For every arrow $\alpha \colon i \to j$ in $Q_{1}$, we write $\overline{\alpha} \colon j \to i$ for its formal inverse. 
It is also convenient to define $\overline{\overline{\alpha}} := \alpha$.
We write $\overline{Q_{1}}$ for the set of formal inverses of arrows of $Q_{1}$. 
We refer to elements of $Q_{1}$ as \emph{direct arrows}, and elements of $\overline{Q_{1}}$ as \emph{inverse arrows}. 

A \emph{walk} is a sequence $\alpha_{1} \dots \alpha_{r}$ of elements of $Q_{1} \cup \overline{Q_{1}}$ such that $t(\alpha_{i}) = s(\alpha_{i + 1})$ and $\alpha_{i + 1} \neq \overline{\alpha_{i}}$ for every $i \in \{1, \dots, r - 1\}$. 
Given a walk $w = \alpha_{1}\dots\alpha_{r}$, we define $\overline{w} = \overline{\alpha_{r}}\dots\overline{\alpha_{1}}$.
When we want to emphasise a vertex $e \in Q_0$ traversed by a walk $w$ we write $e$ in the sequence of $w$. For example, $w = \alpha_1 e \alpha_2$ when $e =t(\alpha_{1}) = s(\alpha_{2})$.

A \emph{string} for $A$ is a walk $w$ in $Q$ avoiding the zero relations in $I$ and such that neither $w$, $\overline{w}$ nor any subword of either is a summand in a relation. 
Given a string $w$, the \emph{string module} $M(w)$ is obtained by replacing every vertex traversed by $w$ by a copy of the field $K$, direct arrows in $w$ by the identity map, and inverse arrows by the identity map in the opposite direction.

A \emph{band} $b$ for $A$ is defined to be a cyclic string such that every power $b^{n}$ is a string, but $b$ itself is not the proper power of some string $w$.
A band $b = \alpha_{1} \dots \alpha_{r}$ can be represented by several different words, such as $\alpha_{2} \dots \alpha_{r}\alpha_{1}$. We refer to these words as \emph{rotations} of the band. 
Given a band $b = \alpha_{1} \dots \alpha_{r}$, an element $\lambda \in K^{\ast}$, and $m \in \mathbb{N}$, the band module $M(b, \lambda, m)$ is obtained from the band $b$ by replacing each vertex traversed by $b$ with a copy of the $K$-vector space $K^{m}$, every direct arrow $\alpha_{i}$ for $i \in \{1, \dots, r - 1\}$ by the identity matrix of dimension $m$, inverse arrows by the identity map in the opposite direction, and the arrow $\alpha_{r}$ by a Jordan block of dimension $m$ with eigenvalue $\lambda$.

We will need to use the descriptions of morphisms between string and band modules from  \cite{CB89,Kra91}, which are as follows. 
Let $w = \alpha_{1} \dots \alpha_{i} \dots \alpha_{j} \dots \alpha_{r}$ be a string with substring $u = \alpha_{i} \dots \alpha_{j}$. 
We say that $u$ is a \emph{submodule substring} if $\alpha_{i - 1}$ is a direct arrow, or $i = 1$, and $\alpha_{j+1}$ is an inverse arrow, or $j = r$. We say that $u$ is a \emph{factor substring} if $\alpha_{i - 1}$ is an inverse arrow, or $i = 1$, and $\alpha_{j + 1}$ is a direct arrow, or $j = r$. 
By \cite{CB89}, given two strings $v, w$ which have a common substring $u$ which is a factor substring of $v$ and such that either $u$ or $\overline{u}$ is a submodule substring of $v$, then there is a map $f_{u}\colon M(v) \to M(w)$. 
As a map of quiver representations, $f_{u}$ is the identity on vertices traversed by $u$, and zero on other vertices.
Moreover, maps of this form constitute a $K$-basis of $\homa(M(v), M(w))$.

Similarly, a $K$-basis of  $\homa(M(v), M(b, \lambda, n))$ (resp.\ $\homa(M(b, \lambda,$ $n), M(v))$) is given by $f_u$ where $u$ is a factor (resp.\ submodule) substring of $v$ and a submodule (resp.\ factor) substring of the infinite string ${}^\infty b^{\infty}$ formed by infinite composition of $b$ with itself.

The situation for maps between band modules is slightly different. Let $M(b, \lambda, n)$ and $M(c, \mu, m)$ be band modules. If either $b \neq c$ or $\lambda \neq \mu$ then a $K$-basis of $\homa( M(b, \lambda, n)$, $M(c, \mu, m))$ is given by maps $f_{(u, \phi)}$ where $u$ is a finite string that is a factor substring of ${}^\infty b^{\infty}$ and a submodule substring of ${}^\infty c^{\infty}$ and $\phi$ is an element of a given basis of $\mathrm{Hom}_{K}(K^n, K^m)$. For more details on these maps, see \cite{STV}.

If $b=c$ and $\lambda = \mu$ then the $K$-basis contains an additional class of maps $f_\psi$ where $\psi \in \mathrm{Hom}_{K}(K^n, K^m)$ and $f_\psi$ is given at a vertex $e$ by a diagonal block $nk \times mk$ matrix where $k$ is the number of occurences of $e$ in the string $b$ and the diagonal blocks are given by $\psi$.

Note that string and band modules give indecomposable modules for every algebra of the form $A=KQ/I$, but the converse is not true in general.
However, for special biserial algebras over an algebraically closed field this is almost true. 
To be more precise, if $A$ is a special biserial algebra, then every indecomposable non-projective $A$-module is either a string module or a band module \cite{BR87, WW85}.

\subsection{Stability spaces and polyhedral geometry} \label{subsec:prelim-stab-spaces-geometry}

We recall the notion of stability from \cite{King94}. We denote the standard inner product on $\mathbb{R}^{n}$ by $\ip{-, -}$. Given $\bv \in \mathbb{R}^{n}$, we say that a non-zero $A$-module $M$ is \emph{$\bv$-semistable} if $\ip{\bv, [M]} = 0$ and $\ip{\bv, [L]} \leqslant 0$ for every proper submodule $L$ of $M$. If $M$ is $\bv$-semistable and $\ip{\bv, [L]} \neq 0$ for all proper submodules $L$ of $M$, we say that $M$ is \emph{$\bv$-stable}. The \emph{stability space} of an $A$-module $M$ is then defined to be \[\D(M) := \{\bv \in \mathbb{R}^{n} \mid M \text{ is }\mathbf{v}\text{-semistable}\}.\] We write $\mathcal{W}_{\bv}$ for the category of $\bv$-semistable $A$-modules. This is a wide subcategory of $\mc A$ by \cite[Proposition 3.24]{BST19}, and hence is an abelian category. The relative simple modules in $\mathcal{W}_{\bv}$ are precisely the $\bv$-stable modules.
These $\bv$-stable modules are always bricks, where a module is a \emph{brick} if its endomorphism algebra is a division ring.

We define the notation \[\Ds(M) := \{\bv \in \mathbb{R}^{\supp_{0}M} \mid M \text{ is }\mathbf{v}\text{-semistable}\}.\] A principle we use throughout is that in order to compute the minimal generating set of $\D(M)$, it suffices to compute $\Ds(M) :=\D(M) \cap \mathbb{R}^{\supp_{0} M}$. This is because if $W$ is a minimal generating set of $\D(M) \cap \mathbb{R}^{\supp_{0}M}$, then $W \cup \{\pm e_{i} \mid i \in Q_{0}\setminus \supp_{0}M\}$ is the minimal generating set of $\D(M)$. See \cite[Lemma~2.5(3)]{Asai}. Hence, it suffices to consider the case where $M$ is sincere.

A \emph{polyhedral cone} $C$ is a subset of $\mathbb{R}^{n}$ of the form \[C = \{ \bv \in \mathbb{R}^{n} \mid A\bv \leqslant 0\},\] where $A$ is a $k \times n$ matrix for some $k$.
It is clear from the definition that $\D(M)$ is a polyhedral cone, since there are only finitely many Grothendieck-group classes of submodules of $M$. The \emph{Minkowski--Weyl Theorem} states that every polyhedral cone admits a description as a non-negative linear span of finitely many vectors \[C = \pspan{\bw_{1}, \dots, \bw_{r}} := \left\lbrace \sum_{i = 1}^{r}\lambda_{i}\bw_{i} \mid \lambda_{i} \geqslant 0 \right\rbrace,\] and conversely that every non-negative linear span of finitely many vectors is a polyhedral cone. If $W = \{\bw_{1}, \dots, \bw_{r}\}$ is minimal with respect to inclusion such that $\pspan{W} = C$, then we say that $W$ is a \emph{minimal generating set for $C$}. In this case, we call $\pspan{\bw_{i}}$ the \emph{rays} of the cone $C$ and $\bw_{i}$ \emph{ray vectors}. A cone is \emph{simplicial} if any minimal generating set is linearly independent.

Recall that a \emph{face} $F$ of a cone $C$ is a subset of $C$ which maximises some linear functional on $C$. That is, \[F = \{\bv \in C \mid \ip{\ba,\bv} \geqslant \ip{\ba, \bw} \; \forall \, \bw \in C\}\] for some $\ba \in \mathbb{R}^{n}$.
A \emph{facet} of $C$ is a face of codimension one. The rays of $C$ are precisely the faces of dimension 1.


\section{Stability spaces of thin representations}\label{sec:thin}

In this section, we show how one may compute the stability space of a thin quiver representation as a non-negative span of a set of vectors. We start by showing how the description of the stability space of a thin module as a non-negative linear span follows from a result of Asai on the facial structure of stability spaces. We then show that this description of the stability spaces of thin modules is related to a result of Geissinger and Stanley describing the facial structure of order polytopes of posets. Finally, we show how the description has a natural application in microeconomics. In later sections, we will show how one can extend the description to string and band modules which are not necessarily thin.

\subsection{The facial structure of a stability space}

We recall the description of the facial structure of a stability space from \cite{Asai}. For a non-zero $A$-module $M$, define \[\simp_{\bv}M := \left\{S \in \brick A \,\middle|\, \parbox{6cm}{$S$  is a simple object appearing in \\ a composition series for  $M$ in $\mathcal{W}_{\bv}$}\right\}\] 
where $\brick A$ denotes the set of bricks in the module category of $A$.  The result \cite[Lemma~2.6]{Asai} states the following. Let $M$ be a non-zero $A$-module, let $\bv \in \D(M)$, and let \[H := \ker \ip{-, [\simp_{\bv}M]} = \{\bx \in \mathbb{R}^{n} \mid \ip{\bx, [S]} = 0 \text{ for all } S \in \simp_{\bv}M\} \subseteq \mathbb{R}^{n}.\] Then
\begin{enumerate}
\item $\D(M) \cap H$ is the smallest face of $\D(M)$ containing $\bv$;
\item for $\bv' \in \D(M)$, we have $\bv' \in \D(M) \cap H$ if and only if $\simp_{\bv}M \subseteq \mathcal{W}_{\bv'}$.
\item for $\bv' \in \D(M)$ and $H' = \ker \ip{-, [\simp_{\bv'}M]}$, we have that $H' \cap \D(M) = H \cap \D(M)$ if and only if $\simp_{\bv}M = \simp_{\bv'}M$.
\end{enumerate}
In other words, the faces of $\D(M)$ are in bijection with filtrations of $M$ by stable modules under some stability condition. One face is contained in another if the filtration associated with the smaller face refines the filtration associated with the larger face.

Describing the facial structure of $\D(M)$ is particularly neat when $M$ is a thin  module.

\begin{lemma}\label{lem:filtrations->stability_conditions}
Let $M$ be a thin module. For every filtration $\mathfrak{F}$ of $M$ with indecomposable factors, there is a stability condition $\bv$ such that $\mathfrak{F}$ is the filtration of $M$ by $\bv$-stable modules.
\end{lemma}
\begin{proof}
Let $\mathfrak{F}$ be the filtration \[0 = M_{0} \subset M_{1} \subset \dots \subset M_{r} = M\] of $M$ and let $S_{i} = M_{i}/M_{i - 1}$ be the factors of the filtration. If $\bv \in \mathbb{R}^{n}$ is such that $S_{i}$ are all $\bv$-stable, then we must have $\ip{\bv, \dimu S_{i}} = 0$ and $\ip{\bv, L} < 0$ for all proper submodules $L$ of $S_{i}$, for all $i$. The bricks $S_{i}$ do not share any composition factors, since $M$ is a thin quiver representation. Hence these inequalities and equations are consistent, which means there does exist such a $\bv \in \mathbb{R}^{n}$ which makes all the $S_{i}$ $\bv$-stable. Consequently, $\mathfrak{F}$ is the filtration of $M$ with $\bv$-stable modules as factors.
\end{proof}

Hence, for a thin module $M$, the faces of $\D(M)$ simply correspond to filtrations of $M$ with indecomposable factors, since all such filtrations arise from stability conditions. We can now use the description of the facial structure of a stability space to compute the stability space of a thin representation as a non-negative span of ray vectors.

\begin{proposition}\label{prop:stability space of thin modules}
Let $M$ be a thin representation of an acyclic quiver $Q$.
Then a minimal generating set of $\D(M)$ is given by \[\{e_{i} - e_{j} \mid i \to j \text{ is an arrow of } Q\}.\]
\end{proposition}
\begin{proof}
We know from \cite[Lemma~2.6]{Asai} and Lemma~\ref{lem:filtrations->stability_conditions} that the rays of a stability space of a thin representation correspond to filtrations of the representation which can only be refined to the filtration by the simple modules. Such filtrations are evidently the ones where all but one factor is simple, and the non-simple factor only has two composition factors. The non-simple factor therefore corresponds to an arrow $i \to j$ of the quiver. The corresponding ray consists of the span of $e_{i} - e_{j}$, since the module $\tcs{$i$\\ $j$}$ must be stable, as must all the other simple modules.
\end{proof}

We hence obtain a neat description of the stability space $\D(M)$ of a thin module $M$ as a non-negative linear span of vectors rather than a space cut out by inequalities.

\begin{example}
Consider the following quiver $Q$.
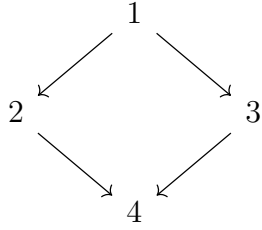
\[
\begin{tikzcd}
& 1 \ar[dr] \ar[dl] & \\
2 \ar[dr] && 3 \ar[dl] \\
& 4 &
\end{tikzcd}
\]
Let $M$ be the representation of $Q$ with dimension vector $\tcs{~1~\\2~3\\~4~}$ and arrows given by the identity. Then by Proposition~\ref{prop:stability space of thin modules} \[\D(M) = \pspan{(1, -1, 0, 0), (1, 0, -1, 0), (0, 1,  0, -1), (0, 0, 1, -1)}\] and, moreover, this is a minimal generating set of $\D(M)$.
\end{example}

We will use the following fact later.

\begin{lemma}\label{lem:simplicial_cone}
Let $Q$ be a quiver such that the underlying graph is a tree. 
Let $M$ be a thin sincere representation of $Q$. 
Then $\D(M)$ is a simplicial cone.
\end{lemma}
\begin{proof}
Let $\sum_{i \to j \in Q_{1}}\lambda_{ij}(e_{i} - e_{j}) = 0$ be a linear dependency for the minimal generating set of $\D(M)$, which we know from Proposition~\ref{prop:stability space of thin modules}. 
Let us split this sum in two depending upon whether the coefficient $\lambda_{ij}$ is positive or negative, giving 
\[\sum_{\substack{i \to j \in Q_{1}\\ \lambda_{ij} > 0}}\lambda_{ij}(e_{i} - e_{j}) = \sum_{\substack{k \to l \in Q_{1}\\ \lambda_{kl} < 0}} - \lambda_{kl}(e_{k} - e_{l}).\] 
Consider expanding the sum on each side so that it is in terms of the $e_{i}$. Since the underlying graph of $Q$ is a tree, $Q$ has no oriented cycles, and so there must be terms $e_{i}$ on each side of the equation which have non-zero coefficients. Suppose $e_{i_{1}}$ has a positive coefficient on the left hand-side of the equation. Hence, there is an arrow $i_{1} \to i_{2}$ in $Q_{1}$ with $\lambda_{i_{1}i_{2}} > 0$. Then, either $e_{i_{2}}$ has a negative coefficient on the left hand side of the equation, or there is some arrow $e_{i_{2}} \to e_{i_{3}}$ with $\lambda_{i_{2}i_{3}} > 0$.

We continue constructing a path $e_{i_{1}} \to e_{i_{2}} \to \dots \to e_{i_{r}}$ in $Q_{1}$ in this way until we reach $e_{i_{r}}$ which has a negative coefficient on the left-hand side of the equation. We then know that $e_{i_{r}}$ has a negative coefficient on the right-hand side of the equation, so that there is an arrow $e_{i_{r + 1}} \to e_{i_{r}}$ in $Q_{1}$ with $\lambda_{i_{r + 1}i_{r}} < 0$. We then proceed in the same way, constructing a path $e_{i_{r}} \leftarrow e_{i_{r + 1}} \leftarrow \dots$ until we find a positive coefficient, in which case we go back over to the left hand side of the equation. The quiver is finite, and so this process must eventually return to a vertex it has already visited before, in which case we have found an unoriented cycle. This is a contradiction, and so our original linear dependency cannot have existed. Note that the set of arrows on the left-hand side of the equation is disjoint from the arrows on the right-hand side, so that we do not traverse backwards along any arrows that we have already traversed forwards.
\end{proof}

We can also describe the facets of the stability space of a thin module, which will become useful later.

\begin{lemma}\label{lem:facets}
Suppose that $M$ is an indecomposable thin module. Then the facets of $\D(M)$ are given by the Cartesian products $\Ds(L) \times \Ds(N)$ for indecomposable submodules $L$ of $M$  with indecomposable quotients $N = M/L$.
\end{lemma}
\begin{proof}
It follows from \cite[Lemma~2.6]{Asai} and Lemma~\ref{lem:filtrations->stability_conditions}---or even simply the definition of a stability space---that the facets of $\D(M)$ are in bijection with filtrations of $M$ with two factors, both of which are indecomposable. Let $L$ then be an indecomposable submodule of $M$ such that the quotient $N = M/L$ is also indecomposable. If we then let $H = \ker\ip{-, [L]} \cap \ker\ip{-, [N]}$, then the corresponding face of $\D(M)$ is $\D(M) \cap H$. Moreover, we have, for $\bv \in \D(M)$, that $\bv \in \D(M) \cap H$ if and only if $\{L, N\} \subseteq \mathcal{W}_{\bv}$. In other words, $\bv \in \D(M) \cap H$ if and only if $L$ and $N$ are both $\bv$-semistable, which proves that $\D(M) \cap H = \Ds(L) \times \Ds(N)$, since if $L$ and $N$ are both $\bv$-semistable, then so is $M$. 
\end{proof}

\subsection{Order polytopes}\label{sect:order_polytopes}

In this section we explain how the result of the previous section on stability spaces of thin modules recovers a result of Geissinger and Stanley on the facial structure of order polytopes of posets.

First we explain some notions concerning posets. Given a poset $P$ and a subset $L \subseteq P$, we say that $L$ is a \emph{lower set} if whenever $x \leqslant y$ and $y \in L$, then $x \in L$. Similarly, given a subset $U \subseteq P$, we say that $U$ is an \emph{upper set} if whenever $x \leqslant y$ and $x \in U$, then $y \in U$. 
A \emph{covering relation} of $P$ is a pair $(x,y) \in P \times P$ such that $x < y$ and for any $z$ with $x \leqslant z \leqslant y$, we have that either $z = x$ or $z = y$. If $(x,y)$ is a covering relation of $P$, then we write $x \lessdot y$. The \emph{Hasse diagram} of $P$ is the directed graph which has the elements of $P$ as its vertices, with an arrow $y \to x$ for every covering relation $x \lessdot y$. A poset $P$ is \emph{connected} if its Hasse diagram is connected.

Let $P$ be a connected poset. Given a partition $\pi = (P_{1}, \dots, P_{n})$ of $P$, one may define a relation $\leqslant_{\pi}$ on $P_{1}, \dots, P_{n}$ by $P_{i} \leqslant_{\pi} P_{j}$ if and only if there exists $p_{i} \in P_{i}$ and $p_{j} \in P_{j}$ such that $p_{i} \leqslant p_{j}$ in $P$. This is the quotient relation associated to the partition $\pi$, see \cite[Section~5]{W21}. Following \cite{Stanley86}, we call a partition $\pi = (P_{1}, \dots, P_{n})$ of $P$ \emph{compatible} if the transitive closure of the quotient relation $\leqslant_{\pi}$ is a partial order. We further call $\pi$ \emph{connected} if $P_{i}$ is connected as a subposet of $P$ for all $i$.

The Hasse diagram of $P$ gives a quiver. This quiver has a representation $M_{P}$ given by a one-dimensional vector space at every vertex, and every arrow corresponding to the identity map. It is clear that submodules of $M_{P}$ correspond precisely to lower sets of $P$ and factor modules correspond precisely to upper sets of $P$. 

\begin{proposition}\label{prop:filt_part}
Filtrations of $M_{P}$ with indecomposable factors are equivalent to connected compatible partitions of $P$. Namely, given a partition $\pi = (P_{1}, \dots, P_{n})$, the corresponding filtration has factors $M_{P_{1}}$, $\dots, M_{P_{n}}$.
\end{proposition}
\begin{proof}
We show that a connected compatible partition gives a filtration with indecomposable factors and \textit{vice versa}. Let $\pi = (P_{1}, \dots, P_{n})$ be a connected compatible partition of $P$. Since the transitive closure quotient relation $\leqslant_{\pi}$ is a partial order, there exists a $P_{i}$ which is minimal with respect to the transitive closure of $\leqslant_{\pi}$. This means that for all $p \in P \setminus P_{i}$ and for all $p_{i} \in P_{i}$, we do not have $p \leqslant p_{i}$. Hence $P_{i}$ is a lower set of $P$, and so $M_{P_{i}}$ is a submodule of $M_{P}$. We may then quotient $M_{P}$ by $M_{P_{i}}$ and inductively build up a filtration of $M_{P}$ with factors $M_{P_{1}},$ $\dots$, $M_{P_{n}}$. These factors are indecomposable since $\pi$ is a connected partition.

Now suppose that we have a filtration of $M_{P}$ with indecomposable factors $M_{P_{1}}, \dots,$ $M_{P_{n}}$. We claim that $\pi = (P_{1}, \dots, P_{n})$ is a connected compatible partition of $P$. It is clear that it is a partition; and it is connected because the factors $M_{P_{1}}, \dots, M_{P_{n}}$ are indecomposable. Suppose, without loss of generality, that the filtration of $M_{P}$ is \[0 = L_{0} \subset L_{1} \subset \dots \subset L_{n} = M_{P}\] where $L_{k}/L_{k - 1} \cong M_{P_{k}}$. We claim that the transitive closure of the quotient relation $\leqslant_{\pi}$ is a subrelation of the total order where $P_{i} < P_{j}$ if and only if $i < j$. For this, it suffices to show that $\leqslant_{\pi}$ is itself a subrelation of this total order. Suppose that $P_{i} <_{\pi} P_{j}$. This means that we cannot have $j < i$, otherwise $M_{P_{j}}$ would not be a submodule of $M_{P}/L_{j - 1}$. Hence $i < j$, and so $\leqslant_{\pi}$ is a subrelation of a total order, which implies that its transitive closure is a partial order, and so $\pi$ is compatible.
\end{proof}

Since $M_{P}$ is a thin representation, all filtrations of $M_{P}$ with indecomposable factors are filtrations by stables under some stability condition by Lemma~\ref{lem:filtrations->stability_conditions}, and so we obtain the following corollary.

\begin{corollary}\label{cor:order_poly}
The faces of $\D(M_{P})$ correspond to connected compatible partitions of $P$.
\end{corollary}

The finest non-trivial connected compatible partitions of $P$ are where all elements of the partition are singletons, bar one, which corresponds to a covering relation. This implies the following corollary, which just reinterprets Proposition~\ref{prop:stability space of thin modules} in the language of posets.

\begin{corollary}\label{cor:cov_rel}
\[\D(M_{P}) = \pspan{e_{y} - e_{x} \mid y \text{ covers } x \text{ in }P}.\]
\end{corollary}

Corollary~\ref{cor:order_poly} recovers the description of the facial structure of order polytopes given by Geissinger and Stanley, which says that the faces of the order polytope $\mathcal{O}(P)$ of a poset $P$ are in bijection with connected compatible partitions of $P$ \cite{Geissinger81,Stanley86}. Corollary~\ref{cor:cov_rel} gives that $\D(M_{P})$ is the dual of the cone $\dual{\D}(M_{P})$ which consists of the points $\bx$ such that $x_{i} \leqslant x_{j}$ for $i \lessdot j$ in $P$. The cone $\dual{\D}(M_{P})$ was called the \emph{monotone polyhedron} of $P$ in \cite{K06}. It was used to define the order polytope of $P$ in \cite{Geissinger81,Stanley86}. Let us extend $P$ to $\hat{P} = \{\hat{0}, \hat{1}\}$, where $\hat{0} \leqslant p$ and $p \leqslant \hat{1}$ for all $p \in P$. Then the \emph{order polytope} $\mathcal{O}(P)$ of $P$ is defined to be the intersection of $\dual{\D}(M_{\hat{P}})$ with the hyperplanes $x_{\hat{0}} = 0$ and $x_{\hat{1}} = 1$. The cone $D(M_{P})$ therefore has the same face lattice as the order polytope $\mathcal{O}(P)$, which explains how Corollary~\ref{cor:order_poly} recovers the description of the faces of the order polytope from Geissinger and Stanley \cite{Geissinger81,Stanley86}.

\subsection{Microeconomic interpretation}

Corollary~\ref{cor:cov_rel} has a natural microeconomic interpretation. Here we will interpret a poset as a set of preferences, and Corollary~\ref{cor:cov_rel} tells us how one may describe the set of advantageous trades for an agent with those preferences.

Consider a set of $n$ goods. We denominate these goods in units such that the goods trade on the market at a 1:1 ratio. That such units exist follows from the assumption of absence of arbitrage. For instance, let us suppose that goods $a$ and $b$ trade on the market at a ratio of 1:1. Consider an additional good $c$, which we denominate in units such that $b$ and $c$ trade at 1:1. If $a$ and $c$ do not then trade at a ratio of 1:1, then there is a possible arbitrage given by first trading $a$ with $b$ and then trading $b$ with $c$.

Now consider an agent whose preferences for the $n$ goods are structured according to a poset $P$. That is, if $a < b$ in $P$, then the agent prefers a unit of good $b$ to a unit of good $a$. In this set-up, let a vector $\mathbf{v} \in \mathbb{R}^{n}$ denote a net amount of each of the respective $n$ goods. A vector $\mathbf{v}$ such that $\sum_{i}v_{i}=0$ then represents a possible trade for our agent, since all the net positions of the goods sum to zero, recalling that the goods trade at a ratio of 1:1.

In this situation, we consider which trades $\mathbf{v}$ leave the agent strictly better off given their poset of preferences $P$. Consider splitting the poset $P$ into two sets $L$ and $U$ such that for every $l \in L$ and $u \in U$, we have that $l \leqslant u$, that is, $u$ is preferred by the agent to $l$. Note that this implies that $L$ must be a lower set and $U$ must be an upper set. It is clear, then, that if a trade $\mathbf{v}$ leaves the agent strictly better off, we must have \[\ip{ \mathbf{v}, L } \leqslant 0 \text{ and } \ip{ \mathbf{v}, U } \geqslant 0.\] If this is not the case, then the preferences of the agent are not satisfied with respect to $U$ and $L$. The agent always prefers goods in $U$ to goods in $L$, so cannot be satisfied if, on balance, they trade goods in $U$ for goods in $L$.

On the other hand, if we have that \[\ip{ \mathbf{v}, L } \leqslant 0 \text{ and } \ip{ \mathbf{v}, U } \geqslant 0\] for every such $U$ and $L$, then the trade satisfies the agent in every possible way. It can be seen that this is equivalent to $M_{P}$ being $\bv$-semistable. Hence, the vectors which make $M_{P}$ semistable correspond precisely to the trades which leave the agent strictly better off. We call such trades \emph{advantageous}.

Corollary~\ref{cor:cov_rel} tells us that the vectors which make $M_{P}$ semistable are precisely those which are sums of $e_{y} - e_{x}$, where $x \lessdot y$ is a covering relation of the poset. Hence we obtain the following result.

\begin{proposition}
The advantageous trades of an agent are precisely the trades which can be reached by a finite number of trades of pairs of goods in which the agent trades a good for one which they prefer.
\end{proposition}

This result is interesting because it gives us two ways of describing the advantageous trades of the agent. On the one hand, we have the definition of them corresponding to semistability, which gives us an intrinsic criterion for recognising whether a vector is an advantageous trade. On the other hand, we have the description as combinations of trades of a good for a preferable one. This way allows us to construct all advantageous trades. Another way of describing the result is that the agent never needs to make a bad trade (or even a neutral one) to reach a desirable end state. All advantageous trades can be reached by trading pairs of goods which leave the agent strictly better off.


\section{Stability spaces of modules}\label{sec:modules}

In this section, we show how one may use Proposition~\ref{prop:stability space of thin modules} to compute the stability spaces of string and band modules as non-negative linear spans. In the next section, we will apply these results to describe how the stability space of a band module is the limit of stability spaces of string modules.

\subsection{String modules}

We now suppose that we have a string module $M(w)$, which possibly has repeated composition factors. 
We furthermore assume that $A$ is a special biserial $K$-algebra, where $K$ is now an algebraically closed field, so that all the submodules of $M$ are either string or band modules.
Let $\thin{w}$ be the thin representation corresponding to $w$. Namely, $\thin{w}$ is the representation of the quiver $Q_{w}$ which has the same shape as $w$ with a copy of the field $K$ at every vertex and identity maps at every arrow.
The aim of this subsection is to prove that one may compute the stability space of $M(w)$ from the stability space of $\thin{w}$; we state this precisely after introducing some helpful notation.

\begin{notation}\label{not: iota map}
For a string module $M(w)$ with $\dimu M(w) = (d_1,\dots, d_n)$, we set $d:= \sum_{i = 1}^n d_i$ so that $\mathcal{D}(\thin{w})$ is a cone in $\mathbb{R}^d = \lspan{ e_i^j \mid 1 \leqslant i \leqslant n, 1 \leqslant j \leqslant d_i }$ and $\mathcal{D}(M(w))$ is a cone in $\mathbb{R}^n = \lspan{ e_i \mid 1 \leqslant i \leqslant n }$. We define the linear map
\begin{align*}
\iota\colon  \mathbb{R}^n & \longrightarrow \mathbb{R}^d \\  e_i  &\longmapsto \sum^{d_i}_{j=1} e_i^j.
\end{align*}
We also introduce notation for the linear subspace where coordinates corresponding to repeated composition factors of $M(w)$ are equal, \[\mathcal{R} := \left\{\bx \in \mathbb{R}^{d} \,\middle|\, \parbox{4.6cm}{\centering $\ip{\bx, e_{i}^{j}} = \ip{\bx, e_{i}^{k}}$,\\ $1 \leqslant i \leqslant n, 1 \leqslant j \leqslant k \leqslant d_{i}$}\right\}.\]
We also use the same notation in the case where $M$ is a band module $M(b, \lambda, r)$.
\end{notation}

\begin{theorem}\label{thm:big-small}
Let $M(w)$ be a string $A$-module over a special biserial $K$-algebra $A$ and let $\thin{w}$ be the thin representation corresponding to the string $w$.
Then \[\iota(\mathcal{D}(M(w))) = \mathcal{D}(\thin{w}) \cap \mathcal{R}.\]
\end{theorem}

Before proving Theorem~\ref{thm:big-small}, we show a series of preliminary results. For this we note that every submodule of $\thin{w}$ defines a submodule of $M(w)$, but not every submodule of $M(w)$ arises from a submodule of $\thin{w}$. We call such submodules \emph{diagonally embedded}. Alternatively, a submodule $L$ of $M(w)$ is diagonally embedded if there is a component of the morphism of representations $L \hookrightarrow M(w)$ which is a matrix with a column containing more than one non-zero entry. One can likewise consider vector subspaces of $M(w)$ as diagonally embedded or not diagonally embedded. We give an example of a diagonally embedded submodule.

\begin{example}
Consider the Kronecker algebra, the path algebra of $1 \rightrightarrows 2$. We consider the module $\tcs{1\\2\:2}$. The submodule \[\tcs{2} \xrightarrow{\sbm{1\\1}}  \tcs{1\\2\:2}\] is diagonally embedded. It can be seen that this submodule does not arise from a submodule of the thin representation of the quiver $2 \leftarrow 1 \rightarrow 2'$. Note that whether or not a given submodule is considered to be diagonally embedded depends upon a choice of basis for $K^{2}$ in our representation $K \rightrightarrows K^{2}$. Such a basis is implicitly fixed when we view this representation as a string module.
\end{example}

In order to prove Theorem~\ref{thm:big-small}, we need to show that, for every submodule $L$ of $M(w)$, there exists a non-diagonally embedded submodule $L'$ such that $\dimu L = \dimu L'$.
This is achieved in the following proposition.
We will use this to show that the defining inequalities of $\iota(\D(M(w)))$ are the same as those of $\mathcal{D}(\thin{w}) \cap \mathcal{R}$.

\begin{proposition}\label{prop: nondiagonal same dim}
Let $M$ be a string module over a special biserial $K$-algebra $A$, and $L$ a submodule of $M$.
Then there exists a submodule $\tilde{L}$ of $M$ such that $\tilde{L}$ is a direct sum of string modules, $\dimu L = \dimu\tilde{L}$, and $\tilde{L}$ does not embed into $M$ diagonally.
\end{proposition}

To prove Proposition~\ref{prop: nondiagonal same dim}, we will use the description of the morphisms between string modules from \cite{CB89,Kra91}, as outlined in Section~\ref{sec:background}. We split into cases depending upon whether we have a string submodule or a band submodule.

\subsubsection{String submodules of string modules}

We begin by proving the following lem\-ma on string submodules of string modules, which will help us to find non-diagonally embedded submodules.

\begin{lemma}\label{lem:string_overlap}
Let $L =M(v)$ and $M = M(w)$ be string modules with $f \colon L \hookrightarrow M$ a monomorphism. 
We may write $m$ as a linear sum $\sum_{j \in J} \lambda_j f_{\sigma_j}$ for some finite set $J$ and $\lambda_{j} \neq 0$, where the map $f_{\sigma_j}$ is induced by $\sigma_j$, with $\sigma_j$ a factor substring of $v$ and submodule substring of $w$. 
Then the following hold:
\begin{enumerate}
\item For every decomposition of $v$ of the form $v=v'xv''$ with $x \in Q_0 \cup Q_1 \cup \overline{Q}_1 $ there exists  $j\in J$ such that $x$ is a subword of  $\sigma_{j}$. In this case we say that $\sigma_j$ \emph{covers} $x$.\label{op:string_overlap:cover}
\item Let $e$ and $e'$ be two consecutive vertices of $v$. Then, for all $j, j' \in J$ such that $\sigma_j$ and $\sigma_{j'}$ cover $e$ and $e'$ respectively, the strings $\sigma_{j}$ and $ \sigma_{j'}$ have a non-empty common subword $\tau$.\label{op:string_overlap:common}
\item Let $j,j' \in J$ be such that $\sigma_j$ and $\sigma_{j'}$ have a common subword $\tau$.
If  $\sigma_j$ and $\sigma_{j'}$ are disjoint as submodule substrings of $w$ then the $K$-subspace of $L$ induced by $\tau$ is diagonally embedded into $M$.
Moreover, $\tau$ is a factor substring of $v$.\label{op:string_overlap:diag}
\end{enumerate}
\end{lemma}

We call the common non-empty subword $\tau$ as in Lemma~\ref{lem:string_overlap}\eqref{op:string_overlap:common} an \emph{overlap} of the strings $\sigma_j$ and $\sigma_{j'}$.

\begin{proof}
First let us consider a decomposition of $v$ of the form  $v=v'xv''$ with $x=e\in Q_0$. 
Suppose that there is no $j \in J$ such that $\sigma_j$ covers $e$. Then we have that $m$ is zero on the vector space at the vertex $e$, which contradicts our hypothesis that $m$ is a monomorphism.
Now we consider the case when $x= \alpha \in Q_1$ is a direct arrow; the case for an inverse arrow will follow dually. 
By the first part of the proof, we know that there exists $j \in J$ such that $\sigma_{j}$ covers $t(\alpha)$. 
Suppose that $\sigma_j$ does not cover $\alpha$. 
Then we have a decomposition \[v= v' \alpha \sigma_j v'''\] which is a contradiction of the fact that $\sigma_j$ is a factor substring of $v$. Thus we have proved~\eqref{op:string_overlap:cover}. We have also proved~\eqref{op:string_overlap:common} since if $\sigma_j$ covers $\alpha$ it also must cover $s(\alpha)$.

To show \eqref{op:string_overlap:diag}, note that the $K$-vector subspace of $L$ generated by $\tau$ is mapped to two different $K$-vector subspaces of $M$.
Hence the $K$-vector subspace of $L$ generated by $\tau$ is diagonally embedded into $M$. 

Finally, let us show that if two strings $\sigma_{j}$ and $\sigma_{j'}$ overlap for $j, j' \in J$, then the overlap $\tau$ is a factor substring of $v$.
Without loss of generality we may assume that there is a decomposition of $\sigma_{j}$ of the form \[\sigma_j = \tau u.\] 
Since $\sigma_{j}$ is a factor substring of $v$ there exists $\alpha, \beta \in Q_1$ such that \[v = v' \overline{\alpha} \sigma_j \beta v'' = v' \overline{\alpha} \tau u \beta v''.\] 
Thus, the arrow $\overline{\alpha}$ to the left of $\tau$ in $v$ is inverse. By a similar argument based on $\sigma_{j'}$,  we deduce that the arrow to the right of $\tau$ in $v$ is direct. Hence $\tau$ is a factor substring of $v$. 
\end{proof}

\begin{lemma}\label{lem:string_mono}
Let $L = M(v)$ and $M = M(w)$ be string modules, with \[f = \sum_{j \in J} f_{\sigma_j}\colon L \to M\] a morphism, where $J$ is a finite set.  Then $f$ is a monomorphism if and only if $\{\sigma_{j} \mid j \in J\}$ covers $v$.
\end{lemma}
\begin{proof}
That $\{\sigma_{j} \mid j \in J\}$ must cover $v$ in the case where $f$ is a monomorphism follows from Lemma~\ref{lem:string_overlap}\eqref{op:string_overlap:cover}.

On the other hand, if $\{\sigma_{j} \mid j \in J\}$ covers $v$, then the definition of the maps $f$ and $f_{\sigma_{j}}$ ensure that $f$ is non-zero on $L_{e}$ for every vertex $e$ occurring in $v$.
\end{proof}

We can now prove the case of Proposition~\ref{prop: nondiagonal same dim} for string submodules.

\begin{lemma}\label{lem: nondiagonal string sub}
Let $M$ be a string module over a special biserial $K$-algebra $A$, and $L$ a string submodule of $M$.
Then there exists a submodule $\tilde{L}$ of $M$ such that $\tilde{L}$ is a direct sum of string modules, $\dimu L = \dimu\tilde{L}$ and $\tilde{L}$ does not embed into $M$ diagonally.
\end{lemma}
\begin{proof}
Let $M = M(w)$, $L = M(v)$, and $f \colon L \mathrel{\hookrightarrow} M$ be the inclusion map.
Then $f$ can be written as a linear sum 
$\sum_{j \in J} \lambda_j f_{\sigma_j}$ for some finite set $J$, where the map $f_{\sigma_j}$ is induced by $\sigma_j$, where $\sigma_j$ is a factor substring of $v$ and submodule substring of $w$.

To begin, we show that there is a subset $R \subset \{\sigma_j \mid j \in J\}$ such that the elements of $R$ cover $v$ and there are no triple overlaps.
By `triple overlaps' we mean portions of the string $v$ which are covered by three or more substrings $\sigma_{j}$ for $j \in J$.
We construct $R$ inductively. 
Namely, start with $R = \emptyset$ and enumerate the vertices of $L$ from left to right $v_0, v_1, \dots, v_N$, and choose $\rho_0 \in \{\sigma_j \mid j \in J\}$ that covers $v_0$ and such that $\max\{ 1 \leqslant i \leqslant N \mid \rho_0 \text{ covers } v_i\}$ is maximal. 
Such a $\rho_{0}$ must exist, by Lemma~\ref{lem:string_overlap}\eqref{op:string_overlap:cover}.
Now, inductively, suppose we have $R= \{ \rho_0, \dots, \rho_j\}$ and that $\rho_j$ does not cover $v_N$. 
Choose a $\rho_{j+1} \in \{\sigma_j \mid j \in J\}$ that covers the last vertex of $\rho_{j}$ and is such that $\max\{ 1 \leqslant i \leqslant N \mid \rho_{j+1} \text{ covers } v_i\}$ is maximal. 
Add $\rho_{j+1}$ to $R$.
Such a $\rho_{j + 1}$ must exist, by Lemma~\ref{lem:string_overlap}\eqref{op:string_overlap:cover} and \eqref{op:string_overlap:common}.
In this way, we eventually construct $R$ such that the elements of $R$ cover $v$ and thus $\sum_{\rho \in R} f_\rho \colon L \to M$ is a monomorphism by Lemma~\ref{lem:string_mono}, noting that we have removed the $\lambda_{j}$ scalars from the original map. 
Since we always choose $\rho_{j}$ to be such that $\max\{ 1 \leqslant i \leqslant N \mid \rho_{j+1} \text{ covers } v_i\}$ is maximal, it is guaranteed that $\rho_{j + 1}$ does not overlap with $\rho_{j - 1}$---otherwise we could have chosen $\rho_{j + 1}$ instead of $\rho_{j}$. 
Hence there are no triple overlaps, as desired. We also note that $\rho_i$ overlaps with $\rho_j$ if and only if $j \in \{i-1, i+1\}$.

We now show that we may inductively resolve the remaining overlaps. 
That is, we will find a submodule of $M$ with the same dimension vector as $L$ where there are no diagonally embedded composition factors. 
Throughout, $\alpha, \beta, \gamma$ and $\delta$ will denote elements of $Q_1$.
We let $u$ be the subword of $v$ covered by $\{\rho_{i} \mid i \geqslant 1\}$. 
Hence there are word decompositions of $v$
\begin{align*}
v & = \rho_0 \beta v'' \\ v& = v' \overline{\alpha} u 
\end{align*} 
 We know that $\rho_0$ and $u$ overlap exactly along a word $\tau$ so there are $\epsilon_0, \epsilon_1 \in Q_1 \cup \overline{Q_1}$ and word decompositions 
\begin{align*}
\rho_0 &= \rho_0' \epsilon_0 \tau \\
u & = \tau \epsilon_1 u'.
\end{align*}
It follows that $\epsilon_0 = \overline{\alpha}$ and $\epsilon_1 = \beta$. 
Now, since $\rho_0$ is a submodule substring of $w$, we have the following decomposition of $w$.
\begin{align*}
w &= w' \gamma \rho_0 \overline{\delta} w''
 \\ &= w' \gamma \rho_0' \overline{\alpha} \tau \overline{\delta} w''.
\end{align*}
Thus, $\rho_0'$ is a submodule substring of $w$.
Hence, $\rho_0'$
defines a morphism of string modules $M(\rho_0') \to M(w)$. 
Consider the morphism of string modules 
\[ \sbm{f_{\rho_0'} & f_u } \colon M(\rho_0') \oplus M(u) \rightarrow M,\] where $f_u= \sum_{i\geq1}f_{\rho_{i}}$. This is a monomorphism by Lemma~\ref{lem:string_mono} since the $\{\rho_i \mid i \geqslant 1\}$ cover $u$.
Moreover, the direct summand $M(\rho'_{0})$ is not diagonally embedded here.
Also, since \[v = \rho_0' \overline{\alpha} u,\] we have that $\dimu L = \dimu M(\rho_0') \oplus M(u)$.
We may inductively proceed in this way, next considering $f_{u}\colon M(u) \to M$, which has strictly fewer morphisms $f_{\rho_{i}}$ as summands than our original morphism $f \colon L \to M$.
By induction, we may resolve all of the overlaps and obtain a monomorphism $\tilde{L} \hookrightarrow M$ where none of the composition factors of $\tilde{L}$ are diagonally embedded, with $\tilde{L}$ a submodule of $M$ such that $\dimu \tilde{L} = \dimu L$.
\end{proof}

\subsubsection{Band submodules of string modules}

The following lemma will help us to find non-diagonally embedded submodules which have the same dimension vectors as diagonally embedded band submodules.

\begin{lemma}\label{lem:band_overlap}
Let $L = M(v, \lambda, m)$ be a band module, $M = M(w)$ be a string module and $f \colon L \hookrightarrow M$ be a monomorphism. 
We may write $f$ as a linear sum $\sum_{j \in J} \lambda_j f_{\sigma_j}$ for some finite set $J$, where $\sigma_j$ is a submodule substring of $w$ and a factor substring of the infinite string $^\infty v ^\infty$, and $\lambda_{j}$ are non-zero.

\begin{enumerate}
\item 
Fix a vertex $e\in Q_0$. 
Then, for each factorisation of $v$ of the form $v=v'ev''$, there are at least $m$ different factorisations of $w$ of the form $w=w'ew''$.\label{op:band_overlap:factorisations}
\item  Then the set  $\{\sigma_j \mid j \in J\}$ can be arranged to cover a  substring of $^\infty v ^\infty$ of the form $v^m$ \textup{(}possibly after rotation of $v$\textup{)}.\label{op:band_overlap:cover}
\end{enumerate}
\end{lemma}

\begin{proof}
\begin{enumerate}
\item 
We first note that the dimension $\dim_K(M_e)$ of the vector space $M_e$ at the vertex $e \in Q$ for $M=M(w)$ is equal to the number of different factorisations of the form $w=w'ew''$.
Meanwhile, the dimension $\dim_K(L_e)$ of the vector space $L_e$ at the vertex $e$ of the band module $L = M(v, \lambda, m)$ is equal to $mt$, where $t$ is the number of factorisations of the form $v=v'ev''$. 
Since $f$ is a monomorphism, we have that $\dim_K(M_e) \geqslant \dim_K(L_e)$ and the result follows.
\item Let $v = v' x e$ for $x \in Q_1 \cup \overline{Q_1}$ and $e \in Q_0$. 
Then the $\{\sigma_j \mid j \in J\}$ can be arranged to cover $v^{m-1}v'$. 
Indeed, suppose that there is a vertex $e_0$ in $v^{m-1}v'$ which is left uncovered. 
Then at the vertex $e_0$,  $m$  has a non-zero kernel which contradicts $m$ being a monomorphism. 
Moreover, there is a word decomposition  $^\infty v ^\infty = u x (v^{m-1}v') x u'$. Observe that, since  the arrow $x$ cannot be both inverse and direct, the set of strings $\{\sigma_j \mid j \in J\}$ must cover more than $v^{m-1}v'$. Hence this set of strings must cover at least $m$ copies of $x$, meaning that a substring of $^\infty v ^\infty$ of the form $v^{m}$ is covered, after possibly rotating $v$.
\end{enumerate}
\end{proof}

\begin{lemma}\label{lem: nondiagonal band sub}
Let $M$ be a string module over a special biserial $K$-algebra $A$, and $L = M(v, \lambda, m)$ a band submodule of $M$.
Then there exists a submodule $\tilde{L}$ of $M$ such that $\tilde{L}$ is a direct sum of string modules, $\dimu L = \dimu\tilde{L}$ and $\tilde{L}$ does not embed into $M$ diagonally.
\end{lemma}
\begin{proof}
We again let $M = M(w)$ and let $f \colon L \hookrightarrow M$ be the inclusion map. 
We will show that there exists a string module $L' = M(u)$ such that $\dimu L' =\dimu L$ and that $L'$ is a submodule of $M$; whence we will be in the situation of Lemma~\ref{lem: nondiagonal string sub}.

As it was said in Lemma~\ref{lem:band_overlap}\eqref{op:band_overlap:cover}, we know that $\iota = \sum_{j \in J} \lambda_j f_{\sigma_j}$ where the set $\{\sigma_j \mid j \in J\}$ covers the string $v^m$.
Moreover, we may choose a rotation of the band $v$ such that the $\sigma_j$ are subwords of $v^m v^\infty$.
Since the string covering the initial segment of $v^{m}v^{\infty}$ must be a factor substring of $^{\infty} v ^{\infty}$, it follows that $v$ is of the form \[ v = v' \overline{\beta} e. \] Let $u$ be the substring $u=v^{m-1}v'$. 
Since the set  $\{\sigma_j \mid j\in J\}$ covers $v^m$, it also covers $u$. Also note that $\dimu M(u)=\dimu L$.

Now, let $j \in J$.
If $\sigma_j$ is a substring of $u$ then it is a factor substring of $u$, by virtue of being a factor substring of $^{\infty} v^{\infty}$. Since $\sigma_j$ is also a submodule substring of $w$, it induces a map $f_{\sigma_j}\colon M(u) \to M(w)$.
In this case, fix $\sigma'_j=\sigma_j$.
Otherwise, $\sigma_j$ is of the form $\sigma_j=\sigma_j'\overline{\beta}\sigma''_j$.
We likewise know in this case that $\sigma_{j}$ is a factor substring of $u$, which gives us that $\sigma_j'$ is a factor substring of $u$ too, noting that there is no arrow of $u$ following $\sigma'_{j}$.
Moreover, since $\sigma_j$ is a submodule substring of $w$ we have that $w=w'\gamma \sigma_j \overline{\delta} w''$.
Then $\sigma'_j$ is a submodule substring of $w$ since $w=w'\gamma \sigma_j'\overline{\beta}\sigma''_j \overline{\delta} w''$.
This implies that $\sigma'_j$ induces a map $f_{\sigma'_j}\colon M(u) \to M(w)$.
Note that the portion of $u$ covered by $\sigma''_{j}$ here must already be covered by some other string $\sigma_{k}$ or $\sigma'_{k}$, due to the rotation we have chosen of $v^{m}$.
As a consequence we can define a map $\iota' \colon M(u) \to M(w)$ as $\iota' = \sum_{j \in J} \lambda_j f_{\sigma'_j}$, which is a monomorphism because the set $\{\sigma'_j \mid j \in J\}$ covers $u$.
This finishes the proof.
\end{proof}

These lemmas now jointly establish Proposition~\ref{prop: nondiagonal same dim}.

\begin{proof}[Proof of Proposition~\ref{prop: nondiagonal same dim}]
Since $A$ is a special biserial algebra, we know that the indecomposable $A$-modules are either string modules, band modules, or projective-injectives. Since $M$ is a string module, $M$ is indecomposable, and so cannot have a projective-injective submodule. We can furthermore assume that the submodule $L$ of $M$ is indecomposable: if $L$ is decomposable then we can consider its indecomposable direct summands one at a time. This gives us that $L$ is either a string submodule, in which case the result follows by Lemma~\ref{lem: nondiagonal string sub}, or a band submodule, in which case the result follows by Lemma~\ref{lem: nondiagonal band sub}.
\end{proof}

\subsubsection{Describing the stability space}

We are now in a position to show how the stability space of the string module $M(w)$ can be described in terms of the stability space of the corresponding thin module $T(w)$.

\begin{proof}[Proof of Theorem~\ref{thm:big-small}]
We claim that if one takes the equations and inequalities defining $\mathcal{D}(\thin{w})$ and sets the coordinates of repeated factors in $M(w)$ equal to each other, then we obtain the equations and inequalities defining $\mathcal{D}(M(w))$.

This follows from Proposition~\ref{prop: nondiagonal same dim} since for every submodule of $M(w)$, we can find a submodule with the same dimension vector which is not diagonally embedded, and therefore corresponds to a submodule of $\thin{w}$.
Thus we obtain every dimension vector of submodules of $M(w)$ by taking submodules of $\thin{w}$ and identifying the coordinates of the repeated factors in $M(w)$. 
Since the inequalities defining the stability space only depend on the dimension vectors of submodules, we are done. 
\end{proof}

In the following example, we illustrate how Proposition~\ref{prop:stability space of thin modules} and Theorem~\ref{thm:big-small} can be applied to calculate the minimal generating set of the stability space of a string module $M$ which is not thin.

\begin{example}
Consider \[M = \tcs{1\\2\\3\\1\\2}\,.\] We first consider the thin representation given by the string of this module, namely \[T = \tcs{1\\2\\3\\1'\\2'}\,.\]
We know from Proposition~\ref{prop:stability space of thin modules} that the vectors generating $\D(T)$ are:
\begin{align*}
(1,-1,0,0,0), \\
(0,1,-1,0,0), \\
(0,0,1,-1,0), \\
(0,0,0,1,-1).
\end{align*}
Hence, we may write any point in $\mathcal{D}(T)$ as
\begin{align*}
n_{1}(1,-1,0,0,0)&+n_{2}(0,1,-1,0,0)\\
&+n_{3}(0,0,1,-1,0)+n_{4}(0,0,0,1,-1) \\
&\qquad= (n_{1},n_{2}-n_{1},n_{3}-n_{2},n_{4}-n_{3},-n_{4}),
\end{align*}
where $n_{1},n_{2},n_{3},n_{4} \in \mathbb{R}_{\geqslant 0}$. Then, by setting the coordinates of the repeated factors in $M$ equal to each other, we obtain
\begin{align*}
n_{1}&=n_{4}-n_{3}, \\
n_{2}-n_{1}&=-n_{4}.
\end{align*}
By adding the two equations together we have that $n_{2}=-n_{3}$. But since these are both non-negative, we must therefore have $n_{2}=n_{3}=0$. Hence $n_{1}=n_{4}$, and so the intersection of $\D(T)$ with the hyperplanes setting the coordinates equal to each other is $(n_{1},-n_{1},0,n_{1},-n_{1}) = \iota(\mathcal{D}(M))$, using the notation from Theorem~\ref{thm:big-small}. Hence we have that $\mathcal{D}(M) = \pspan{(1,-1,0)}$.
\end{example}

We now develop this method for computing $\D(M(w))$ further, by showing how one can find the ray vectors of $\D(M(w))$ without solving equations. This will be useful for proving results in Section~\ref{sec:w&cspecial}.

\begin{definition}
Let $M(w)$ be a string $A$-module and let $S$ be a minimal generating set of $\D(\thin{w})$. 
We say that a sum $\bw = \sum_{\bv \in S}\lambda_{\bv}\bv$ is an \emph{admissible sum} if $\bw$ lies in all the hyperplanes which set the relevant elements of $\thin{w}$ equal to each other. 
We say that $\bw$ is a \emph{minimal admissible sum} if it is an admissible sum and there is no non-trivial admissible sum of vectors from a proper subset of $S$.
\end{definition}

For the next result, we recall the linear map $\iota$ from Notation~\ref{not: iota map}.

\begin{proposition}\label{prop:min_adm_sum}
Let $M(w)$ be a string $A$-module. A vector $\bv \in \mathcal{D}(M(w))$ is a ray vector if and only if $\iota(\bv)$ is a minimal admissible sum of ray vectors in $\D (\thin{w})$.
\end{proposition}

\begin{proof}
We show that every ray vector of $\iota(\D(M(w)))$ is a minimal admissible sum of ray vectors of $\D(\thin{w})$ and, conversely, that every minimal admissible sum of ray vectors of $\D(\thin{w})$ is a ray vector of $\iota(\D(M(w)))$. Let $B$ be the minimal generating set of $\D(\thin{w})$.

Let $\bw$ be a ray vector of $\iota(\D(M(w)))$. 
Since $\iota(\D(M(w)))$ is a subset of $\D(\thin{w})$ by Theorem~\ref{thm:big-small}, we have that $\bw = \sum_{\bv \in S}\lambda_{\bv}\bv$, where $S \subseteq B$. 
We may also assume that $\lambda_{\bv}$ is non-zero for all $\bv \in S$. 
Since $\bw \in \iota(\D(M(w)))$, it is clear that this is an admissible sum. 
Suppose that it is not a minimal admissible sum. 
Then there exists a proper subset $S' \subset S$ such that $\bw' = \sum_{\bv \in S'}\lambda'_{\bv}\bv \in \iota(\D(M(w)))$. 
Choose the maximal $\mu > 0$ such that $\bw'' = \bw - \mu\bw'$ is a non-negative linear sum of elements of $S$. 
Since $\mu$ is maximal, we must have that the coefficient of $\bv'$ in $\bw''$ is zero, for some $\bv' \in S'$. 
Moreover, $\bw'' \in \iota(\D(M(w)))$, since it is a non-negative linear combination of elements of $B$ and lies in the right hyperplanes. 
However, this means that $\bw = \mu\bw' + \bw''$, with $\bw'$ and $\bw''$ elements of $\iota(\D(M(w)))$ which are not scalar multiples of each other. But then $\bw$ is not a ray vector of $\iota(\D(M(w)))$.

Conversely, let $\bw = \sum_{\bv \in S}\lambda_{\bv}\bv$ be an admissible sum, where $S \subseteq B$ and $\lambda_{\bv}$ are all non-zero. 
If this is not a ray vector of $\iota(\D(M(w)))$, then there exist $\bw', \bw'' \in \iota(\D(M(w)))$ such that $\bw = \bw' + \bw''$, with $\bw'$ and $\bw''$ not scalar multiples of each other. 
Since $\iota(\D(M(w)))$ is a subset of $\D(\thin{w})$, we have $\bw' = \sum_{\bv \in S'}\lambda'_{\bv}\bv$ and $\bw'' = \sum_{\bv \in S''}\lambda''_{\bv}\bv$, where we can assume that the scalars $\lambda'_{\bv}$ and $\lambda''_{\bv}$ are non-zero. 
If $S'$ and $S''$ are proper subsets of $S$, then it is immediate that $\bw$ is not a minimal admissible sum. 
Otherwise, we similarly let $\mu$ be maximal such that $\bw' - \mu\bw''$ is a non-negative linear sum.
This lies in $\iota(\D(M(w)))$ and the coefficient of some $\bv' \in S''$ must be zero. 
Hence $\bw' - \mu\bw''$ is an admissible sum of vectors from a proper subset of $S$, which means $\bw$ is not a minimal admissible sum.
\end{proof}

We can use this proposition to compute the stability space $\D(M(w))$, as we illustrate in the following examples.

\begin{example}
We consider again $M$ and $T$, as in the previous example.
We take each generating vector of $\mathcal{D}(T)$ in turn and add multiples of other generating vectors until the coordinates of the repeated factors of $M$ are equal, giving an admissible sum. 
By adding as few vectors as possible, we ensure that it is a minimal admissible sum.

For instance, if we start with $(1,-1,0,0,0)$, in order to get an admissible sum we must make the first and fourth coordinates equal, as well as the second and fifth. 
We can achieve this by adding the vector $(0,0,0,1,-1)$ to obtain $(1,-1,0,1,-1)$.

If we now consider the vector $(0,1,-1,0,0)$, there is no other vector we can add to equate the second and fifth coordinates without rendering the first and fourth no longer equal. 
Hence we ignore this vector. A similar argument applies to $(0,0,1,-1,0)$.

Since we have effectively already considered $(0,0,0,1,-1)$, the only generating vector we obtain is $(1,-1,0,1,-1)$, which corresponds to $(1,-1,0)$ once the coordinates are identified. Hence, we again obtain $\D(M) = \pspan{(1, -1, 0)}$
\end{example}

\subsection{Band modules}

We now study the stability spaces of band modules. 
Recall that a \emph{homogeneous tube} is a component of an Auslander--Reiten quiver of the form $\mathbb{Z}A_{\infty}/\langle \tau \rangle$.
An important property of band modules is that they form homogeneous tubes in the Auslander--Reiten quiver of the algebra. 
We start by showing a general result about the stability spaces of modules in the same homogeneous tube.

\begin{proposition}
Let $M, N$ be finitely generated $A$-modules  in the same homogeneous tube of rank 1 of the Auslander--Reiten quiver of $A$. Then $\D(M) = \D(N)$.
\end{proposition}

\begin{proof}
We can suppose without loss of generality that $M$ is the module sitting at the mouth of the tube. 
Then, every module $N$  in the same tube can be filtered by~$M$.

Let $\mathbf{v}\in\D(M)$. 
As we noted in Section~\ref{subsec:prelim-stab-spaces-geometry}, it was shown in \cite[Proposition~3.24]{BST19} that the category of $\mathbf{v}$-semistable modules is a wide subcategory of $\mc A$. Recall that a wide subcategory is one which is closed under kernels, cokernels, and extensions. Since $N$ is filtered by $M$  and $M$ is a $\bv$-semistable module, so is $N$.
Hence $\mathbf{v}\in\D(N)$.

Conversely, let $\mathbf{v}\in\D(N)$.
Since $N$ is filtered by $M$ we have the following two short exact sequences.
\begin{align*}
&0\to N' \to N \to M \to 0\\
&0\to M \to N \to N'' \to 0
\end{align*}
This gives the existence of an endomorphism $f\colon N \to N$ such that Im$f \cong M$. 
Since the category of $\mathbf{v}$-semistable objects is wide in $\mc A$ \cite[Proposition~3.24]{BST19} and $N$ is $\mathbf{v}$-semistable by hypothesis, so is $M$.
\end{proof}

As a direct consequence of the previous we obtain the following corollary.

\begin{corollary} \label{cor: band dimension}
Let $M(b, \lambda, 1)$ be a band module. Then we have $\D(M(b, \lambda, r))=\D(M(b, \lambda, 1))$ for all $r\in \mathbb{N}$.
\end{corollary}

\begin{remark} \label{rem: band scalars}
Observe that, for $\lambda$, $\mu \in K^\ast$, there is a bijection between submodules of $M(b, \lambda, 1)$ and submodules of $M(b, \mu, 1)$ which matches submodules which have the same composition factors. 
It follows that $\D({M(b, \lambda, 1)})=\D({M(b, \mu, 1)})$.
\end{remark}

We may ``de-diagonalise'' proper string and band submodules of a band module as we did for string modules in Proposition~\ref{prop: nondiagonal same dim}. By Corollary~\ref{cor: band dimension}, it suffices to consider band modules of the form $M(b, \lambda, 1)$.

\begin{proposition} \label{prop: nondiagonal same dim band} 
Let $M= M(b, \lambda, 1)$ be a band module and $L$ a proper submodule of $M$ that is a string or a band module. Then there exists a submodule $\tilde{L}$ of $M$ such that $\tilde{L}$ is a direct sum of string modules, $\dimu L = \dimu\tilde{L}$ and $\tilde{L}$ has no composition factors that embed diagonally into $M$.
\end{proposition}

\begin{proof}
The proof is analogous to the proof of Proposition~\ref{prop: nondiagonal same dim}. 
Indeed, if $L= M(v)$ is a string module then the map $L \hookrightarrow M$ is a linear sum $\sum_{j \in J} \lambda_j f_{\sigma_j}$ for some finite set $J$, where $\sigma_j$ is a factor substring of $v$ and a submodule substring of the infinite string $^\infty b ^\infty$ and the argument of Lemma~\ref{lem: nondiagonal string sub} applies.

Now we consider the case where $L = M(c, \mu, m)$ is a band module. If $c=b$ and $\mu = \lambda$ then, since $L$ is a submodule of $M$, we must have that $m=1$ and $L$ is not a proper submodule. Hence we may assume that either $c\neq b$ or $\mu \neq \lambda$. Then, the map $L \hookrightarrow M$ is a linear sum  $\sum_{j \in J} \lambda_j f_{(\phi_j, u_j)}$ using the description of morphisms between band modules from Section~\ref{sec:background}. We then may proceed as in Lemma~\ref{lem:band_overlap} and Lemma~\ref{lem: nondiagonal band sub}.
\end{proof}

Consequently, when working over a special biserial algebra, we obtain a version of Theorem~\ref{thm:big-small} for band modules. We once again make use of Notation~\ref{not: iota map}.

\begin{theorem} \label{thm:big-small band}
Let $M = M(b, \lambda, r)$ be a band module over a special biserial $K$-algebra $A$ and let $\thin{b}$ be the thin representation corresponding to the band $b$.
Then \[\iota(\mathcal{D}(M)) = \mathcal{D}(\thin{b}) \cap \mathcal{R}.\]
\end{theorem}

\begin{remark}
The procedure for finding non-diagonally embedded submodules which have the same dimension vectors as diagonally embedded submodules also extends to tree modules in the sense of \cite{CB89}, provided one assumes that every submodule of the tree module is either a tree module or a band module.
\end{remark}

\subsubsection{Bands, strings, and posets}

In the manner of Section~\ref{sect:order_polytopes}, we may consider the quivers given by strings and bands as posets, which allows us to make the following observations.

\begin{proposition} \label{prop:bandstab}
Let $N = M(b, \lambda, r)$ be a band module over a special biserial algebra $A$ such that $M(b, \lambda, 1)$ is a thin module. 
Then for any two  distinct strings $v, v'$ that are obtained from $b$ by removing an arrow, we have that \[\mathcal{D}(N) = \pspan{ \mathcal{D}(M(v)), \mathcal{D}(M(v'))}.\]
Moreover $\D(N) = \D(M(v))$ if and only if $v$ and $b$ are isomorphic as posets.
\end{proposition}

\begin{proof}
By Corollary~\ref{cor: band dimension} and Remark~\ref{rem: band scalars}, for the task of computing $\D(M(b, \lambda, r))$ we may assume that $\lambda = 1$ and $r = 1$. 
Observe that  $\dimu M(b,1,1) = \dimu M(v) = \dimu M(v')$ and that every covering relation in the poset defined by the band $b$ is contained in one of the posets induced by $v$ or $v'$. 
Hence the first statement follows from  Theorems~\ref{thm:big-small} and~\ref{thm:big-small band} and Corollary~\ref{cor:cov_rel}.

For the second part of the statement, without loss of generality suppose that $b=v q_1\alpha q_2$ with $q_1, q_2 \in Q_0$.
Then the arrow $\alpha$ induces the relation $q_1 \geqslant q_2$ in the poset generated by $b$, which in turn induces the vector $e_{q_1} - e_{q_2} \in \D(N)$.
Hence, we have that $e_{q_1} - e_{q_2} \in \D(M(v))$ if and only if $q_1 \geqslant q_2$ in the poset structure induced by $v$. 
In particular, this is equivalent to say that the poset induced by $b$ and $v$ are isomorphic.
\end{proof}

A direct consequence of the previous proposition is the following. (Compare \cite[Lemma~5.2]{Asai}.)

\begin{corollary}
Let $N = M(b, \lambda, r)$ be a band module over a special biserial algebra and let $v$ be a string obtained from $b$ by removing one arrow. 
Then $\D(M(v))\subseteq \D(N)$.
\end{corollary}

\section{The wall-and-chamber structure of special biserial algebras}\label{sec:w&cspecial}

In the previous section we described the stability spaces of string and band modules over special biserial algebras.
This is not the first time that the wall-and-chamber structure of special biserial algebras has been studied. 
In fact, it has been shown in \cite{AokiYurikusa} that, for these algebras, the union of all the chambers is a dense set in~$\mathbb{R}^n$.
It was shown in \cite{BST19, Asai} that for any finite-dimensional algebra each chamber is surrounded by $n$ different walls.
Moreover, it follows from the results in \cite{Treffinger2019} that if the stability space of a band is a wall, then this wall is not contained in the closure of a single chamber.
Consequently, we can conclude that if the stability space of a band module is a wall, then it is in the closure of an infinite family of chambers. 
This is equivalent to saying that if the stability space of a band module is a wall, then this wall is the limit of infinite families of walls determined by string modules.
In this section we construct certain infinite families of strings for any given band.
We then use the results of the previous section to show that the stability space of the band module decomposes as a union of cones which are the limits of the stability spaces of these string modules.

Our set-up for the section is as follows. Let $\Lambda = KQ/I$ be a special biserial algebra with $b = \alpha_1 \alpha_2 \dots \alpha_n$ a band. We will first assume that $M(b, \lambda, 1)$ is a thin module and then use this to deduce the general case. We will also be considering the restriction of the stability space of a module to the support of the module, so that we will be writing $\D$ for $\Ds$. We enumerate the vertices of $b$ from left to right as $v_1, \dots, v_n, v_{n+1}$ (so that $v_1 = v_{n+1}$). Hence, we have that arrow $\{s(\alpha_{i}), s(\alpha_{i + 1})\} = \{v_{i}, v_{i + 1}\}$. Then, by Proposition~\ref{prop:stability space of thin modules}, we know that $\D(M(b,1,1)) = \pspan{ S} $ where $S$ consists of the vectors $e_i - e_j$ where $i \to j$ is an arrow in the thin representation $\thin{b}$. We fix $0 \leqslant m < n$ and set $w = \alpha_1 \alpha_2 \dots \alpha_m$ (for $m=0$ we set $w$ to be the trivial word at the vertex $v_1$).

 \begin{remark} \label{rem:band g vect}
Since we are working over a special biserial algebra, the $g$-vector of $M(b, 1, 1)$ can be explicitly computed:
\begin{equation} \label{eqn:band g vector}
\g{M(b, 1, 1)} = \sum_{i=1}^{n} \varepsilon_i e_i,
\end{equation}
where \[ \varepsilon_i=  \left\{ \begin{array}{ll} 1, & v_i \in \tp(M(b, 1, 1)) \\ -1, & v_i \in \soc(M(b, 1, 1)) \\ 0, & \textrm{else.} \end{array} \right. \] Note that $\g{M(b, \lambda, t)} = t\g{M(b, 1, 1)}$ and that $\g{M(b, 1, 1)} \in \D (M(b, 1, 1))$ since $\g{M(b, 1, 1)}$ is the sum of all the vectors $e_{i} - e_{j}$ corresponding to all direct (or all inverse) arrows in $b$. 

We may also explicitly compute the $g$-vector of $M(b^rw)$. Indeed, the $g$-vector of $M(w)$ is given by the following formula
\begin{equation} \label{eqn:string g vector}
 \g{M(w)} = a e_1 + b e_{m + 1} + c e_{m + 2} + d e_n + \sum_{i=2}^{m} \varepsilon_i e_i 
\end{equation} 
where 
\[
a = \left\{ \begin{array}{ll} 1, & \alpha_1 \textrm{ is direct} \\ 0, & \alpha_1 \textrm{ is inverse} \end{array} \right. , 
\qquad
b = \left\{ \begin{array}{ll} 1, & \alpha_m \textrm{ is inverse and } m \neq 0 \\ 0, & \textrm{else} \end{array} \right. ,
\]\[
c = \left\{ \begin{array}{ll} 0, & \alpha_{m+1} \textrm{ is inverse or } m=0 \\ -1, &  \textrm{else} \end{array} \right. ,
\quad
d = \left\{ \begin{array}{ll} 0, & \alpha_n \textrm{ is inverse} \\ -1, & \alpha_n \textrm{ is direct} \end{array} \right. ,
\]
and the $\varepsilon_i$ are as above. 
It is then also not hard to see that the following holds
\begin{equation} \label{eqn: brw g vector}
\g{M(b^rw)} = r \g{M(b,1,1)} + \g{M(w)}. 
\end{equation} 
\end{remark}

Thus the $g$-vector of $M(w)$ encodes the information of the orientation of the arrows $\alpha_i$ for $i \in \{1,m,m+1,n\}$. We make use of Notation~\ref{not: iota map} again.

\begin{proposition} \label{prop:band string 2}
 We have that 
\[ \D(M(bw)) = \pspan{S' \cup \{\bx\} },\] where $S' \subset S$ is the set of generating vectors $\D(M(b, 1, 1))$ corresponding to arrows of $\thin{b}$ except for $\alpha_{m + 1}$ and $\alpha_{n}$, and 
\[
\bx = \left\{ \begin{array}{ll} \g{M(w)}, & w \textrm{ is a factor substring of } b \\ -\g{M(w)^{\mathrm{op}}}, & w \textrm{ is a submodule substring of } b \\ 0, & \textrm{else.} \end{array} \right.
\] Here $M(w)^{\mathrm{op}}$ is the string module corresponding to the walk $w$ over the opposite algebra.
\end{proposition}

\begin{proof}
 Observe that $\D(\thin{bw})$ is a cone in \[\mathbb{R}^{n+m+1} = \lspan{e^1_1, e^1_2, \dots, e_n^{1}, e^2_1, e^2_2, \dots, e^2_{m+1}}\] and that by Theorem~\ref{thm:big-small} we know that 
 \[\iota(\D(M(bw))) = \D(\thin{bw}) \cap \lspan{e_1^1 + e_1^2, e_2^1 + e_2^2, \dots, e_{m + 1}^1 + e_{m + 1}^2, e_{m + 2}^{1}, \dots, e_{n}^{1}}.\] Moreover, by Proposition~\ref{prop:min_adm_sum}, the ray vectors of $\iota(\D(M(bw)))$ correspond to minimal admissible sums of ray vectors of $\D(\thin{bw})$. Now, clearly every element of $\iota(S')$ is already a minimal admissible sum, and so $\pspan{\iota(S')} \subset \D(M(bw))$.  
 Indeed, for $1\leq k \leq m$ the sum of the vectors corresponding to $\alpha_k^1$ and $\alpha_k^2$ is a minimal admissible sum; and the vector corresponding to $\alpha_l^1$ is a minimal admissible sum for $m+1<l<n$. 
 
We first verify the claim for when $w$ is a factor substring of $b$. Observe that this is the case precisely when $\alpha_n$ is inverse and $\alpha_{m+1}$ is direct.  
 Let $\bv$ be a minimal admissible sum such that $\bv \not\in \pspan{\iota(S')}$, which will be of the form 
 \[ \bv = \lambda (e^2_1 - e^1_n) + \mu (e_{m+1}^1 - e_{m+2}^1) + \bv' \] with $\lambda, \mu \geq 0$ and $\bv'$ in the subspace generated by the vectors corresponding to the arrows $\alpha^1_1, \dots, \alpha^1_{m}, \alpha^1_{m+2}, \dots, \alpha_{n-1}^1, \alpha^2_1, \dots, \alpha^2_m$. We denote by $\lambda^i_j$ the positive coefficient in $\bv'$ of the generating vector corresponding to $\alpha^{i}_{j}$. Note that by the minimality of $\bv$ we immediately have that $0 = \lambda^1_{m+2} = \lambda^1_{m+3} = \dots = \lambda^1_{n-1}$ and that at most one of $\lambda^1_j$ and $\lambda^2_j$ can be non-zero for $1 \leq j \leq m$; else $\bv'$ contains an admissible subsum corresponding to the vectors in $S'$. 
 
 We denote the coefficient of $e_i^j$ in $\bv$ by $\varepsilon_i^j$. Note that since $\bv$ is an admissible sum, we must have that $\varepsilon_i^1 = \varepsilon_i^2$ for $1 \leq i \leq m$. 
 
 First suppose that $\lambda = 0 = \mu$, so that $\bv = \bv'$, and let $1 \leq k \leq m$ be minimal such that $\lambda^1_k$ is non-zero. By scaling, we may assume that $\lambda^1_k =1$. We deal with the situation where $\alpha_k$ is inverse, whence the case where it is direct follows by a similar argument. As $\alpha_k$ is inverse and $k$ is minimal, $\varepsilon_k^1 = -1$. Therefore $\varepsilon_k^2 $ must also be $-1$. We deduce that $\alpha_{k-1}$ must be direct and that $\lambda^2_{k-1} = 1$, since we cannot have that $\lambda_{k}^{2}$ is non-zero whilst $\lambda_{k}^{1}$ is non-zero. Indeed, if $\alpha_{k-1}$ is inverse, then $\varepsilon_k^2 = \lambda^2_{k-1}-\lambda^2_k = \lambda^2_{k-1} - 0 \geq 0$. Now, as the $\varepsilon_i^1$ are $0$ for $1\leq i <k$ by assumption, we must have that $\varepsilon_{k - 1}^{2} = 0$, which forces $\alpha_{k - 2}$ to be direct, with $\lambda_{k - 1}^{2} = 1$. By repeating this argument, we conclude that all arrows $\alpha_i$ are direct with $\lambda^2_i = 1$ for $1\leq i <k$. We obtain a contradiction, since $\alpha_{n}$ is inverse, so we must have $\varepsilon_{1}^{2} = 1$, which is in conflict with the fact that $\varepsilon_{1}^{1} = 0$. If $k = 1$, we obtain that $\varepsilon_{1}^{1} = -1$ and $\varepsilon_{2}^{2} = 1$, which is likewise a contradiction.
 
 Thus, we may suppose that $\lambda =1$. We first consider the case where $\alpha_1$ is direct, that is, $v_1 \in \tp(M(b,1,1))$; if $\alpha_1$ is inverse the argument is similar. 
 As $\bv$ is admissible, $\varepsilon_1^2 = \lambda + \lambda^2_1 = 1+ \lambda^2_1$, which must equal $\varepsilon^1_1 = \lambda^1_1 \geq 0$. Since $\lambda^2_1$ and $\lambda^1_1$ cannot both be non-zero, we deduce that $\lambda^1_1 = 1$ and $\lambda_1^2 =0$. 
 
 Whilst the following arrow remains direct, we continue with similar reasoning. Indeed, suppose that $\alpha_2$ is direct. Then $\varepsilon^2_2 = \lambda^2_2 \geq 0$ and $\varepsilon^1_2 =  -\lambda^1_1 + \lambda^1_2 = -1 + \lambda^1_2$ must be equal and thus, $\lambda_2^2 =0$ and $\lambda^1_2 = 1$. 
 
 At some point, an inverse arrow, $\alpha_s$, may be reached. Then, $\varepsilon^1_s = -\lambda_{s-1}^1 - \lambda^1_s = -1 - \lambda^1_s \leq -1$ and $\varepsilon^2_s = - \lambda^2_s $ are equal, from which we deduce that $\lambda^2_s = 1$ and $\lambda_s^1 = 0$. Then, whilst the subsequent arrows are inverse, we repeat the dual argument to the above one for sequences of direct arrows. 
 
  
By continuing this argument we reach the following: For $1 \leq j \leq m$, if $\alpha_j$ is direct, then $\lambda^1_j = 1$ and $\lambda_j^2 = 0$; and if $\alpha_j$ is inverse, then $\lambda^1_j =0$ and $\lambda_j^2 = 1$. 

It remains to consider the coefficients of $e_{m+1}^1$ and $e_{m+1}^2$. If $\alpha_m$ is direct, then $\varepsilon^1_{m+1} = -\lambda^1_m + \mu = -1 + \mu$ and $\varepsilon^2_{m+1} = -\lambda^2_m = 0$ and therefore $\mu = 1$. If $\alpha_m$ is inverse, then $\varepsilon^1_{m+1} = \lambda^1_m + \mu = \mu$ and $\varepsilon_{m+1}^2 = \lambda^2_m = 1$ and thus $\mu =1$. 

Observe that now  for $1 \leq j \leq m$ the coefficients $\varepsilon_j^1$ (and therefore also the $\varepsilon_j^2$) are $1$ (resp.\ $-1$) precisely when $v_j$ is in the top (resp.\ socle) of  $M(bw)$ and that $\varepsilon_n^1$  and $\varepsilon_{m+1}^1$ are $-1$. In light of Remark~\ref{rem:band g vect}, we have obtained precisely the $g$-vector of $M(bw)$. 
 
 If we instead begin with the assumption that $\mu =1$, then by similar reasoning, we reach the same admissible sum as above. Hence, for this case we are done.

 
 The case when $w$ is a submodule substring of $b$ is dual. So we must now deal with the case when either $\alpha_n$ and $\alpha_{m + 1}$ are both either inverse (with the dual situation of both being direct). In this case, we begin with the same logic as above. The same argument shows that $\lambda = 0 = \mu$ results in no minimal admissible sum not contained in the set $S'$. However, in the case of $\lambda =1$, 
after following the same strategy as above and reaching that the $\lambda^1_k$ are $1$ for $\alpha_k$ direct and $0$ otherwise (and vice versa for the $\lambda^2_k$), 
 we reach a contradiction as $\varepsilon^2_{m+1}$ and $\varepsilon^1_{m+1}$ cannot be equal. 
 Indeed,   say $\alpha_m$ is also inverse. Then  $\lambda_m^1 = 0$ and $\lambda_m^2 = 1$ and so
$\varepsilon_{m+1}^1 = \lambda^1_m - \mu = 0 - \mu \leq 0$ and $\varepsilon^2_{m+1}= \lambda_m^2 = 1$ which is a contradiction. 
Dually, if $\alpha_m$ is direct, then $\varepsilon_{m+1}^1 = -\lambda^1_m - \mu = -1 - \mu \leq -1$ and $\varepsilon^2_{m+1}= \lambda_m^2 = 0$ which is also a contradiction. 
\end{proof}

We show in the next result that Proposition~\ref{prop:band string 2} and the explicit expressions of the $g$-vectors from Remark~\ref{rem:band g vect} allow us to describe a minimal generating set of $\D(M(b^rw))$.

\begin{proposition}\label{prop:stringgenerators_general}
\[ \D(M(b^{r}w)) = \pspan{S' \cup \{\bx\} },\] where $r \geqslant 1$, $S' \subset S$ is the set of generating vectors $\D(M(b, 1, 1))$ corresponding to arrows of $\thin{b}$ except for $\alpha_{m + 1}$ and $\alpha_{n}$, and 
\[
\bx = \left\{ \begin{array}{ll} \g{M(b^{r-1}w)}, & w \textrm{ is a factor substring of } b \\ -\g{M(b^{r - 1}w)^{\mathrm{op}}}, & w \textrm{ is a submodule substring of } b \\ 0, & \textrm{else.} \end{array} \right.
\] Here $M(b^{r-1}w)^{\mathrm{op}}$ is the string module corresponding to the walk $\overline{b^{r - 1} w}$ over the opposite algebra.
\end{proposition}

\begin{proof}
All elements of $\iota(S')$ can be seen as minimal admissible sums of ray vectors of $\D(T(b^{r}w))$. This is because the vertices $v_{1}, \dots, v_{m}, v_{m + 1}$ are covered $r + 1$ times by $b^{r}w$, whilst the vertices $v_{m + 2}, \dots, v_{n}$ are only covered $r$ times. Given an arrow in $b$ between two vertices repeated $r + 1$ times, taking the sum of the $r + 1$ generating vectors corresponding to the arrows in $T(b^{r}w)$ gives a minimal admissible sum. The same applies for arrows in $b$ between vertices repeated $r$ times. The only arrows in $b^{r}w$ between vertices which are not repeated the same number of times are $v_{m + 1} \xrightarrow{\alpha_{m + 1}} v_{m + 1}$ and $v_{n} \xrightarrow{\alpha_{n}} v_{1}$. Hence, $\pspan{S'} \subseteq \D(M(b^{r}w))$ by Proposition~\ref{prop:min_adm_sum}.

Firstly we will prove the rest of statement for when $w$ is a factor substring of $b$, that is,  $\alpha_n$ is inverse and $\alpha_{m+1}$ is direct.  
The method of proof is induction on $r$; the base case, $r=1$, was shown in Proposition~\ref{prop:band string 2}.
Let $\bz \in \D(M(w^r))$ be a ray vector and we write $\bg := \bg^{M(b,1,1)}$. 

\textbf{Step 1: Either $\bz \in S'$ or we may assume $\ip{\bz, (1,1,\dots,1)} = -1 $.} Consider the string $\sigma = \alpha_1\alpha_2\dots \alpha_{n-1}$ then $b^rw = \sigma \alpha_n b^{r-1}w$. By assumption, we have that the arrow $\alpha_n$ is inverse and thus $M(\sigma)$ is a submodule of $M(w^r)$. Since $\dimu M(\sigma) = (1,1,\dots,1)$ we have that $\ip{\bz, (1,1,\dots,1)} \leq 0$. Let us show that  $\ip{\bz, (1,1,\dots,1)} = 0$ then $\bz \in S'$. Indeed, if this were the case we will show that then $\bz \in \D (M(b,1,1))$ from which we will verify our claim.  Notice that $\dimu M(b,1,1) = (1,1,\dots,1)$. Let $L$ be a proper submodule of $M(b,1,1)$. By Proposition~\ref{prop: nondiagonal same dim band} we may assume that $L = M(\tau)$ is a string module and $\tau$ is a submodule substring of $b$. Hence $\tau$ is also a submodule substring of $b^rw$ and so $\ip{\bz, \dimu L } \leq 0$. This shows that if $\ip{\bz, (1,1,\dots,1)} =0$ then $\bz \in \D(M(b,1,1)) = \pspan{ S' \cup \{ (e_1 - e_n), (e_{m+1} -e_{m+2})\}}$. So we have that there are $\lambda, \mu >0$ and $\by \in \pspan{S'}$ such that 
\[ \bz = \lambda(e_1-e_n) + \mu (e_{m+1} -e_{m+2}) + \by. \]
  Now,
 \begin{align*}
 0 &= \ip{\bz, \dimu M(b^rw) } 
 \\&= \ip{\lambda(e_1-e_n) + \mu (e_{m+1} -e_{m+2}) + \by, (r+1, \dots, r+1, r, \dots, r) } 
 \\&= \lambda \ip{e_1-e_n, (r+1, \dots, r+1, r, \dots, r)} 
 \\& \hphantom{= } + \mu \ip{e_{m+1} -e_{m+2}, (r+1, \dots, r+1, r, \dots, r)} + 0 
 \\&= \lambda + \mu
\end{align*}  so $\lambda = \mu =0$ and $\bz = \by \in \pspan{S'}$ and since $\bz$ is a ray vector, $\bz \in S'$. Therefore if $\bz \not\in \pspan{S'}$ then $\ip{\bz, (1,1,\dots,1)} <0$ and so, after scaling, we may assume that $\ip{\bz, (1,1,\dots,1)} = -1$.

\textbf{Step 2: $\bz -\bg \in \D(M(b^{r-1}w))$.} We verify this directly. Firstly, we see that
\begin{align*} 
&\ip{\bz-\bg, \dimu M(b^{r-1}w} \\ =& \ip{ \bz,(r, \dots, r, r-1, \dots, r-1)} - \ip{\bg,(r, \dots, r, r-1, \dots, r-1) } 
\\ =& \left[ \ip{\bz, (r+1, \dots, r+1, r, \dots, r)} - \ip{\bz, (1,1,\dots,1)} \right] 
\\ & - \left[ \ip{\bg, (1, \dots, 1, 0,\dots,0)} + (r-1)\ip{\bg,(1,\dots,1)} \right]
\\  =& 0 +1 -1 +0 = 0, 
\end{align*}
where for the computation of the third inner product, we used the fact that $w$ has one more vertex in the top of $b$ than the socle of $b$, since it is a factor substring of $b$; and for the fourth that $\g{M(b,1,1)} \in \D(M(b,1,1))$ by  Remark~\ref{rem:band g vect}.

Now let $L$ be a proper submodule of $M(b^{r-1}w)$. By Proposition~\ref{lem: nondiagonal string sub}, we may assume up to dimension vector, that $L$ is a string module corresponding to a submodule substring $\sigma$ of $b^{r-1}w$ with word decomposition $b^{r-1}w = u \sigma u'$ which exhibits $\sigma$ as a submodule substring of $b^{r-1}w$. 

If $u$ is the empty word then, as $u'$ must be non-empty,  then the word  decomposition $b^r w =  u \sigma u' \alpha_{m+1} \dots \alpha_n \alpha_1 \dots \alpha_m $ shows $\sigma$ as a submodule substring of $b^rw$. Thus $\ip{\bz, \dimu L } \leq0$. Observe that $\ip{\bg, \dimu L} =0$ and so $\ip{\bw - \bg, \dimu L }\leq0$.

In the case that the word $u'$ is non-empty we may assume $u$ is non-empty too, else we can repeat the above argument. Now, the word decomposition $b^rw=bu \sigma u'$ shows $\sigma$ as a submodule substring of $b^rw$ so $\ip{\bz, \dimu L} \leq 0$. Note that, in this case, $\ip{\bg, \dimu L} \in \{0,-1\}$ so we must show that $\ip{\bz, \dimu L } <-1$.  Observe that, since $u$ and $u'$ are non-empty, there is a factor substring $\tau$ of $b^rw$ such that $\dimu L + \dimu M(\tau) = k(1,\dots, 1)$ for some $1\leq k \leq r$. Indeed, if $\alpha_a$ is the last arrow of $u$ and $\alpha_b$ is the first arrow of $u'$ then set $\tau = \alpha_{b+1} \alpha_{b+2} \dots \alpha_n \alpha_1 \dots \alpha_{a-1}$. Therefore, since $\ip{\bz, \dimu M(\tau)} \geq 0$ we have that 
\[ \ip{\bz, \dimu L} = k\ip{\bz,(1,1,\dots, 1)} - \ip{\bz, \dimu M(\tau)} \leq -k. \] We deduce that $\ip{\bz- \bg, \dimu L} \leq 0$ and so $\bz \in \D (M(b^{r-1}w))$. 

\textbf{Step 3: Inductive argument.}
By Step 2 we have that $\bz-\bg \in \D(M(b^{r-1}w))$. By our inductive hypothesis, we have that 
\[ \bz - \bg = \lambda (\bx+ (r-2)\bg) + \by \]  where $\bx= \g{M(w)}$, $\lambda >0$ and $\by \in \pspan{S'}$.  Now
\begin{align*}
0 &= \ip{\bz, \dimu M(b^rw)} 
\\ & = \ip{  \lambda (\bx+ (r-2)\bg) + \bg + \by, (r+1, \dots, r+1, r, \dots, r) } 
\\ &= \lambda \ip{\bx, (r+1, \dots, r+1, r, \dots, r)}  
\\& \hphantom{=} + (\lambda (r -2) +1) \ip{ \bg, (r+1, \dots, r+1, r, \dots, r)} + 0 \\
&= \lambda(1-r) + (\lambda (r -2) +1)(1) \\
&= 1 -  \lambda.
\end{align*}  
Thus $\lambda =1$ and $\bz = \bx+(r-1)\bg + \by$. Since, we assumed $\bz$ was a ray vector, we deduce that $\by =0$ which completes the proof. 

The proof of the case when $w$ is a submodule substring of $b$ follows by dualising the above. Finally we deal with the case when $w$ is neither in a factor nor submodule substring of $b$. Let us assume that the arrows $\alpha_n $ and $\alpha_{m+1}$ are both direct, whence the case when they are both inverse  will follow by dualising the argument. Let $\bz \in \D(M(w^r))$ be a ray vector.

\textbf{Step 1': $\ip{\bz, (1,1,\dots,1)} =0$.}   Consider the words 
 \begin{align*}
 \sigma &= \alpha_1 \alpha_2 \dots \alpha_{n-1} \\
 \sigma '&= \alpha_{m+2} \alpha_{m+3} \dots \alpha_n \alpha_1 \dots \alpha_m 
 \end{align*} then the word decompositions $b^rw = \sigma \alpha_n b^{r-1}w$ and $b^rw = b^{r-1}w\alpha_{m+1} \sigma'$ exhibit $M(\sigma)$ and $M(\sigma')$ as a factor module and submodule of $M(b^rw)$ respectively. Since $\dimu M(\sigma) = \dimu M(\sigma') = (1,1,\dots,1)$ we deduce that $\ip{z, (1,1,\dots,1)}=0$. 
 
 \textbf{Step 2': $\bz \in \D(M(bw))$.} Again we show this directly. First,
 \begin{align*} 
\ip{\bz, \dimu M(bw)} =& \ip{ \bz,(2,\dots, 2, 1, \dots, 1)}  
\\ =&  \ip{\bz, (r+1, \dots, r+1, r,\dots, r)} - (r-1) \ip{\bz, (1,1,\dots,1)}  
\\ =&  0 +0  = 0. 
\end{align*}
Now let $L$ be a proper submodule of $M(bw)$. Then, similarly to Step 2, we see that there there is a submodule of $b^rw$ with the same dimension vector as $L$. Thus $\ip{\bz, \dimu L} \leq 0$ which shows that $\bz \in \D(M(bw))$.

\textbf{Step 3': Finish the proof.}
By the base case, we know that $\D(M(bw)) = \pspan{S'}$ and $\bz \in \pspan{S'}$ completes the proof.

\end{proof}

In the following lemma, we show how the stability space of a band module can be covered by simplicial cones. We will later show how these simplicial cones can be given as the limits of stability spaces of string modules.

\begin{lemma}\label{lem:cover}
Let $A$ be a special biserial algebra and let $b = \alpha_1 \dots \alpha_n$ be a band in $A$ such that $M(b, 1, 1)$ is a thin module. Denote by $S$ a minimal generating set for $\D(M(b, 1, 1))$ and let $\bg = \g{M(b, 1, 1)}$ be the $g$-vector of $M(b, 1, 1)$. Then 
\begin{equation}\label{eq:band_stab_space_formula}
\D(M(b, 1, 1)) = \bigcup_{\bv, \bw \in S, \bv \neq \bw} \pspan{(S \setminus \{\bv, \bw\}) \cup \{\bg\}}.
\end{equation}
\end{lemma}

\begin{proof}
Recall from Remark~\ref{rem:band g vect} that $\bg \in \D(M(b, 1, 1))$. In fact, as discussed in Remark~\ref{rem:band g vect}, the sum of all the ray vectors of $\D(M(b, 1, 1))$ gives a multiple of $\bg$, which implies that $\bg$ is in the interior of $\D(M(b, 1, 1))$.

First note that it is clear that \[\D(M(b, 1, 1)) \supseteq \bigcup_{\bv, \bw \in S, \bv \neq \bw} \pspan{(S \setminus \{\bv, \bw\}) \cup \{\bg\}},\] since $\D(M(b, 1, 1)) = \pspan{S \cup \{\bg\}}$.
Now choose a point $\bx \in \D(M(b, 1, 1))$. We wish to show that $\bx$ lies in the right hand side of \eqref{eq:band_stab_space_formula}. This is clear if $\bx \in \pspan{\bg}$, so we assume that this is not the case. Hence, consider a line segment connecting $\bx$ with the boundary of $\D(M(b, 1, 1))$ which passes through $\pspan{\bg}$. Let $\by$ be the point where this line segment passes through the boundary of $\D(M(b, 1, 1))$.

The point $\by$ thus lies in a facet of $\D(M(b, 1, 1))$. By Lemma~\ref{lem:facets}, the facet containing $\by$ must be equal to $\D(L) \times \D(M(b, 1, 1)/L)$ for some indecomposable submodule $L$ of $M(b, 1, 1)$ with indecomposable quotient. Since $M(b, 1, 1)$ is thin, we must have that the supports of $L$ and $M(b, 1, 1)/L$ must be trees. By Lemma~\ref{lem:simplicial_cone}, we have that $\D(L)$ and $\D(M(b, 1, 1))$ must both be simplicial cones. Hence, if $L$ has dimension $r$, then the number of generating vectors of the facet containing $\by$ must be $(r - 1) + (n - r - 1)$. Hence, let the facet containing $\by$ be generated as a non-negative linear span of $S \setminus \{\bv, \bw\}$, where $\{\bv, \bw\} \subseteq S$ with $\bv \neq \bw$. Consequently, we have that $\bx \in \pspan{(S \setminus \{\bv, \bw\}) \cup \{\bg\}}$, whence the result follows. 
\end{proof}

In our next result we denote by $B(\bx, \epsilon)$ the open ball of $\mathbb{R}^n$ with centre $\bx$ and radius $\epsilon > 0$. We show that every point in the stability space of a thin band module can be approximated by the stability spaces of families of strings.

\begin{proposition}\label{prop:limit}
Let $A$ be a special biserial algebra and let $b=\alpha_1 \dots \alpha_n$ be a band in $A$ such that $M(b, 1, 1)$ is thin.
Then, for every $\bx \in \D(M(b, 1, 1))$, there exists an infinite family of strings $\{b^{r}w \mid r \in \mathbb{N}\}$ such that, for every $\epsilon > 0$, there is an $r_{\epsilon} \in \mathbb{N}$ such that $B(\bx, \epsilon) \cap \D(M(b^{s}w)) \neq \varnothing$ for every $s \geqslant r_{\epsilon}$.
\end{proposition}

\begin{proof}
By Lemma~\ref{lem:cover}, there are $\bv, \bw \in S$ such that $\bv \neq \bw$ and $\bx \in \pspan{S' \cup \bg}$, where $S' := S \setminus \{\bv, \bw\}$ and $\bg$ is the $g$-vector of $\D(M(b, \lambda, 1))$. Note that, generically, we can assume that $\codim \pspan{S' \cup \{\bg\}} = \codim \D(M(b, 1, 1))$. We can rotate the band $b$ such that $\bv$ corresponds to the arrow $\alpha_{n}$ and $\bw$ corresponds to the arrow $\alpha_{m + 1}$. Then let $w$ be the string $\alpha_{1} \dots \alpha_{m}$. By Proposition~\ref{prop:stringgenerators_general}, we have that \[\D(M(b^{r}w)) = \pspan{S' \cup \{\by + (r-1)\bg\}} = \pspan{S' \cup \left\{\frac{1}{r-1}\by + \bg\right\}}.\] Since $\bx \in \pspan{S' \cup \{\bg\}}$ and $\frac{1}{r-1}\by + \bg \to \bg$ as $r \to \infty$, we have that for every $\epsilon > 0$, there is an $r_{\epsilon} \in \mathbb{N}$ such that $B(\bx, \epsilon) \cap \D(M(b^{r}w)) \neq \varnothing$ for every $s \geqslant r_{\epsilon}$.
\end{proof}

We can then apply this proposition to prove the analogous result in the case where we do not have a thin module.

\begin{theorem}\label{thm:limit_nonthin}
Let $A$ be a special biserial algebra and let $b = \alpha_1 \dots \alpha_n$ be an arbitrary band in $A$.
Then, for every $\bx \in \D(M(b, \lambda, t))$, there exists an infinite family of strings $\{b^{r}w \mid r \in \mathbb{N}\}$ such that, for every $\epsilon > 0$, there is an $r_{\epsilon} \in \mathbb{N}$ such that $B(\bx, \epsilon) \cap \D(M(b^{s}w)) \neq \varnothing$ for every $s \geqslant r_{\epsilon}$.
\end{theorem}

\begin{proof}
First recall that, by Corollary~\ref{cor: band dimension}, we have $\D(M(b, \lambda, t)) = \D( M(b, 1, 1))$, so we only need to prove the statement for the case where $\lambda =1$ and $t=1$.
This result can be deduced from Proposition~\ref{prop:limit}. By Theorem~\ref{thm:big-small band} and Notation~\ref{not: iota map}, we have $\iota(\D(M(b, 1, 1))) = \D(T(b)) \cap \mathcal{R}$. Hence, given a point $\bx \in \D(M(b, 1, 1))$, there is a corresponding point $\iota(\bx) \in \D(T(b))$. We know that we can approach with a family of strings $\{{b'}^{r}w' \mid r \in \mathbb{N}\}$ for $T(b)$, by Proposition~\ref{prop:limit}. Consider the corresponding family of strings $\{b^{r}w \mid r \in \mathbb{N}\}$ of $A$. By Theorem~\ref{thm:big-small}, we know that $\iota(\D(M(b^{r}w))) = \D(M({b'}^{r}w')) \cap \mathcal{R}$. Hence, since $\{{b'}^{r}w' \mid r \in \mathbb{N}\}$ approaches $\bx$, we must have that for every $\epsilon > 0$, there is an $r_{\epsilon} \in \mathbb{N}$ such that $B(\bx, \epsilon) \cap \D(M(b^{r}w)) \neq \varnothing$ for every $s \geqslant r_{\epsilon}$.
\end{proof}

The main result of \cite{STV} states that a special biserial algebra is $\tau$-tilting finite if and only if no band module is a brick. 
In particular this implies that in the wall-and-chamber structure of a $\tau$-tilting finite special biserial algebra there are no walls which are determined by band modules. 
In the next result we show that, if there is a wall determined by a band module, then this wall can be seen as the limit of a family of walls determined by string modules. 

\begin{corollary}
Let $A$ be a special biserial algebra and let $b$ be a band such that $\D(M(b, \lambda, t))$ is a wall in the wall-and-chamber structure of $A$. 
Then, for every $x \in \D(M(b, \lambda, t))$, there exists an infinite family of strings $\{w_i \mid i \in \mathbb{N}\}$ such that $\D(M(w_i))$ is a wall and  there is an $i_x \in \mathbb{N}$ such that $B(x, \epsilon)\cap \D(M(w_j)) \neq \varnothing$ for every $j \geqslant i_x$.
\end{corollary}

\begin{proof}
This follows from Theorem~\ref{thm:limit_nonthin} and the observation that it suffices to consider the strings whose stability spaces have the same codimension as that of the band. This is because Theorem~\ref{thm:limit_nonthin} gives us countably many string modules such that every point in the stability space of the band module is a limit of points in the stability spaces of the string modules. If we did not have infinitely many string modules whose stability space had the same codimension as that of the band, then uncountably many string modules would be needed in order to have every point in the $\D(M(b, \lambda, t))$ approached by points in the stability spaces of string modules. Hence, if the stability space of the band module has codimension 1, then, for a given point in the stability space of the band module, there must be infinitely many string modules whose stability spaces have codimension 1 which approach it.
\end{proof}

We finish by illustrating our results with the following example. Note how the convergence of stability spaces can be understood more clearly by viewing the stability spaces in terms of positive spans of sets of vectors, rather than in terms of inequalities.

\begin{example}\label{ex:convergence}
We consider the path algebra of the quiver 
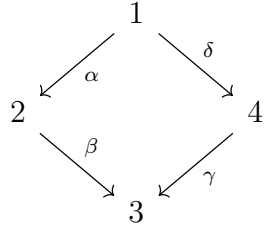
\[
\begin{tikzcd}
& 1 \ar[dl,"\alpha"] \ar[dr,"\delta"] & \\
2 \ar[dr,"\beta"] && 4 \ar[dl,"\gamma"] \\
& 3 &
\end{tikzcd}
\] and the band $b = \alpha\beta\overline{\gamma}\overline{\delta}$. We have by Proposition~\ref{prop:stability space of thin modules} that the stability space $\D(M(b, 1, 1))$ is equal to $\pspan{e_{1} - e_{2}, e_{2} - e_{3}, e_{4} - e_{3}, e_{1} - e_{4}}$. The cross-section of this stability space is shown in Figure~\ref{fig:ex:convergence}. It can be seen that this stability space divides into four parts. We shall find a family of strings whose stability spaces converge to each of these four parts. Indeed, let
\begin{align*}
w_{1} &= \overline{\gamma}\overline{\delta}\alpha\beta, & w_{2} &= \overline{\gamma}\overline{\delta}\alpha\beta\overline{\gamma} , \\
w_{3} &= \alpha\beta\overline{\gamma}\overline{\delta}, & w_{4} &= \overline{\delta}\alpha\beta\overline{\gamma}\overline{\delta}.   \\
\end{align*}
The corresponding string modules are 
\begin{align*}
M(w_{1}) &= %
	{\tcs{\phantom{3}\phantom{2}1\phantom{4}\phantom{3}\\%
	\phantom{3}2\phantom{1}4\phantom{3}\\%
	3\phantom{2}\phantom{1}\phantom{4}3}},%
 & M(w_{2}) &= %
 	{\tcs{\phantom{3}\phantom{4}1\phantom{2}\phantom{3}\phantom{4}\\%
 	\phantom{3}4\phantom{1}2\phantom{3}4\\%
 	3\phantom{4}\phantom{1}\phantom{2}3\phantom{4}}}, \\
M(w_{3}) &= %
	{\tcs{1\phantom{2}\phantom{3}\phantom{4}1\\%
	\phantom{1}2\phantom{3}4\phantom{1}\\%
	\phantom{1}\phantom{2}3\phantom{4}\phantom{1}}},%
 & M(w_{4}) &= %
 	{\tcs{\phantom{4}1\phantom{2}\phantom{3}\phantom{4}1\\%
 	4\phantom{1}2\phantom{3}4\phantom{1}\\%
 	\phantom{4}\phantom{1}\phantom{2}3\phantom{4}\phantom{1}}}.
\end{align*}
By Proposition~\ref{prop:stringgenerators_general}, the stability spaces of these modules are
\begin{align*}
\D(M(w_{1})) &= \pspan{e_{1} - e_{4}, e_{1} - e_{2}, e_{2} + e_{4} - e_{3}},\\
\D(M(w_{2})) &= \pspan{e_{1} - e_{2}, e_{4} - e_{3}, e_{1} + e_{2} - e_{3}},\\
\D(M(w_{3})) &= \pspan{e_{2} - e_{3}, e_{4} - e_{3}, e_{1} - e_{2} - e_{4}},\\
\D(M(w_{4})) &= \pspan{e_{1} - e_{4}, e_{1} - e_{2}, e_{1} - e_{2} - e_{3}}.
\end{align*}
In order to define the families of strings converging to the different parts of the stability space of the band, we will need to define a few different rotations of the band $b$, namely
\begin{align*}
b_{1} &= w_{1}, & b_{2} &= b_{1} \\
b_{3} &= w_{3}, & b_{4} &= \overline{\delta}\alpha\beta\overline{\gamma}.
\end{align*}
Our four families of strings are then $b_{1}^{r}w_{1}$, $b_{2}^{r}w_{2}$, $b_{3}^{r}w_{3}$, and $b_{4}^{r}w_{4}$. We can compute the stability spaces of the corresponding string modules using Proposition~\ref{prop:stringgenerators_general}, obtaining
\begin{align*}
\D(M(b_{1}^{r}w_{1})) &= \pspan{e_{1} - e_{4}, e_{1} - e_{2}, \frac{1}{r - 1}(e_{2} + e_{4} - e_{3}) + e_{1} - e_{3}},\\
\D(M(b_{2}^{r}w_{2})) &= \pspan{e_{1} - e_{2}, e_{4} - e_{3}, \frac{1}{r - 1}(e_{1} + e_{2} - e_{3}) + e_{1} - e_{3}},\\
\D(M(b_{3}^{r}w_{3})) &= \pspan{e_{2} - e_{3}, e_{4} - e_{3}, \frac{1}{r - 1}(e_{1} - e_{2} - e_{4}) + e_{1} - e_{3}},\\
\D(M(b_{4}^{r}w_{4})) &= \pspan{e_{1} - e_{4}, e_{1} - e_{2}, \frac{1}{r - 1}(e_{1} - e_{2} - e_{3}) + e_{1} - e_{3}}.
\end{align*}
Therefore, as $r \to \infty$, we have that
\begin{align*}
\D(M(b_{1}^{r}w_{1})) &\to \pspan{e_{1} - e_{4}, e_{1} - e_{2}, e_{1} - e_{3}},\\
\D(M(b_{2}^{r}w_{2})) &\to \pspan{e_{1} - e_{2}, e_{4} - e_{3}, e_{1} - e_{3}},\\
\D(M(b_{3}^{r}w_{3})) &\to \pspan{e_{2} - e_{3}, e_{4} - e_{3}, e_{1} - e_{3}},\\
\D(M(b_{4}^{r}w_{4})) &\to \pspan{e_{1} - e_{4}, e_{1} - e_{2}, e_{1} - e_{3}}.
\end{align*}
And hence for each part of the stability space of the band module we obtain a family of string modules whose stability spaces converge to it.
\begin{figure}[H]
\caption{An illustration of the cross-section of the stability space $\D(M(b, 1, 1))$ from Example~\ref{ex:convergence}}\label{fig:ex:convergence}
\[
\begin{tikzpicture}[scale=2]

\coordinate(c14) at (0,0);
\coordinate(c12) at (0,2);
\coordinate(c43) at (2,2);
\coordinate(c23) at (2,0);

\node(n14) at (c14) {$\bullet$};
	\node [below left = 0.5mm of n14] {$e_{1} - e_{4}$};
\node(n12) at (c12) {$\bullet$};
	\node [above left = 0.5mm of n12] {$e_{1} - e_{2}$};
\node(n43) at (c43) {$\bullet$};
	\node [above right = 0.5mm of n43] {$e_{4} - e_{3}$};
\node(n23) at (c23) {$\bullet$};
	\node [below right = 0.5mm of n23] {$e_{2} - e_{3}$};

\draw (c14) -- (c12);
\draw (c12) -- (c43);
\draw (c43) -- (c23);
\draw (c23) -- (c14);

\draw[dotted] (c12) -- (c23);
\draw[dotted] (c14) -- (c43);

\node(n13) at (1,1) {$\bullet$};
	\node [below = 0.5mm of n13] {$e_{1} - e_{3}$};

\end{tikzpicture}
\]
\end{figure}
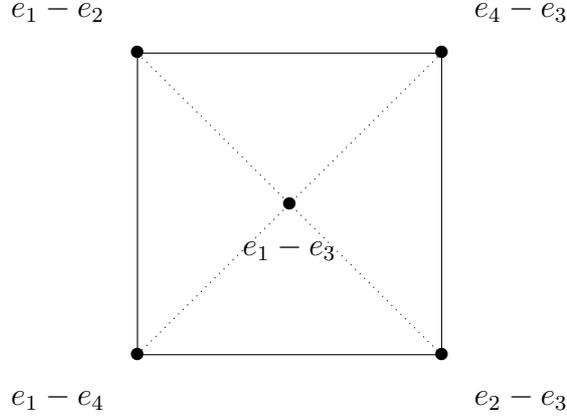

\end{example}

\begin{remark}
Note that, in Example~\ref{ex:convergence}, there also exist the strings
\begin{align*}
w'_{1} &= \beta\overline{\gamma}\overline{\delta}, & w'_{2} &= \overline{\gamma}\overline{\delta}\alpha, \\
w'_{3} &= \overline{\delta}\alpha\beta, & w'_{4} &= \alpha\beta\overline{\gamma}.
\end{align*}
The corresponding string modules are
\begin{align*}
M(w'_{1}) &= %
	{\tcs{\phantom{2}\phantom{3}\phantom{4}1\\%
	2\phantom{3}4\phantom{1}\\%
	\phantom{2}3\phantom{4}\phantom{1}}},%
 & M(w'_{2}) &= %
 	{\tcs{\phantom{2}1\phantom{4}\phantom{3}\\%
 	2\phantom{1}4\phantom{3}\\%
 	\phantom{2}\phantom{1}\phantom{4}3}}, \\
M(w'_{3}) &= %
	{\tcs{\phantom{3}\phantom{2}1\phantom{4}\\%
	\phantom{3}2\phantom{1}4\\%
	3\phantom{2}\phantom{1}\phantom{4}}},%
 & M(w'_{4}) &= %
 	{\tcs{1\phantom{2}\phantom{3}\phantom{4}\\%
 	\phantom{1}2\phantom{3}4\\%
 	\phantom{1}\phantom{2}3\phantom{4}}}.
\end{align*}
It is then easy to compute using Proposition~\ref{prop:stability space of thin modules} that
\begin{align*}
\D(M(w'_{1})) &= \pspan{e_{2} - e_{3}, e_{4} - e_{3}, e_{1} - e_{4}},\\
\D(M(w'_{2})) &= \pspan{e_{1} - e_{2}, e_{1} - e_{4}, e_{4} - e_{3}},\\
\D(M(w'_{3})) &= \pspan{e_{2} - e_{3}, e_{1} - e_{2}, e_{1} - e_{4}},\\
\D(M(w'_{4})) &= \pspan{e_{1} - e_{2}, e_{2} - e_{3}, e_{4} - e_{3}}.
\end{align*}

One can note that these stability spaces then cover $\D(M(b, 1, 1))$. However, this does not have the significance that the convergence in Example~\ref{ex:convergence} has. These strings are exceptional strings, meaning that they lie in tubes, whereas the families of strings in Example~\ref{ex:convergence} are non-exceptional.
\end{remark}

\bibliography{bibliography} 

\begin{thebibliography}{10}

\bibitem{AIR14}
Takahide Adachi, Osamu Iyama, and Idun Reiten.
\newblock {$\tau$}-tilting theory.
\newblock {\em Compos. Math.}, 150(3):415--452, 2014.

\bibitem{Amiot}
Claire Amiot.
\newblock Cluster categories for algebras of global dimension 2 and quivers
  with potential.
\newblock {\em Ann. Inst. Fourier (Grenoble)}, 59(6):2525--2590, 2009.

\bibitem{AokiYurikusa}
Toshitaka Aoki and Toshiya Yurikusa.
\newblock Complete special biserial algebras are $g$-tame, 2020.

\bibitem{Asai}
Sota Asai.
\newblock The wall-chamber structures of the real {G}rothendieck groups.
\newblock {\em Adv. Math.}, 381:107615, 2021.

\bibitem{Asai22}
Sota Asai.
\newblock Non-rigid regions of real {G}rothendieck groups of gentle and special
  biserial algebras.
\newblock {\em arXiv preprint arXiv:2201.09543}, 2022.

\bibitem{BridgelandScat}
Tom Bridgeland.
\newblock Scattering diagrams, {H}all algebras and stability conditions.
\newblock {\em Algebr. Geom.}, 4(5):523--561, 2017.

\bibitem{BST19}
Thomas Br\"{u}stle, David Smith, and Hipolito Treffinger.
\newblock Wall and chamber structure for finite-dimensional algebras.
\newblock {\em Adv. Math.}, 354:106746, 31, 2019.

\bibitem{BMRRT}
Aslak~Bakke Buan, Robert Marsh, Markus Reineke, Idun Reiten, and Gordana
  Todorov.
\newblock Tilting theory and cluster combinatorics.
\newblock {\em Adv. Math.}, 204(2):572--618, 2006.

\bibitem{BR87}
Michael C.~R. Butler and Claus~Michael Ringel.
\newblock Auslander--{R}eiten sequences with few middle terms and applications
  to string algebras.
\newblock {\em Comm. Algebra}, 15(1-2):145--179, 1987.

\bibitem{CCS}
Philippe Caldero, Fr\'ed\'eric Chapoton, and Ralf Schiffler.
\newblock Quivers with relations arising from clusters ({$A_n$} case).
\newblock {\em Trans. Amer. Math. Soc.}, 358(3):1347--1364, 2006.

\bibitem{CB89}
William~W. Crawley-Boevey.
\newblock Maps between representations of zero-relation algebras.
\newblock {\em J. Algebra}, 126(2):259--263, 1989.

\bibitem{DWZ}
Harm Derksen, Jerzy Weyman, and Andrei Zelevinsky.
\newblock Quivers with potentials and their representations. {I}. {M}utations.
\newblock {\em Selecta Math. (N.S.)}, 14(1):59--119, 2008.

\bibitem{Geissinger81}
Ladnor Geissinger.
\newblock The face structure of a poset polytope.
\newblock In {\em Proceedings of the {T}hird {C}aribbean {C}onference on
  {C}ombinatorics and {C}omputing ({B}ridgetown, 1981)}, pages 125--133. Univ.
  West Indies, Cave Hill Campus, Barbados, 1981.

\bibitem{GHKK}
Mark Gross, Paul Hacking, Sean Keel, and Maxim Kontsevich.
\newblock Canonical bases for cluster algebras.
\newblock {\em J. Amer. Math. Soc.}, 31(2):497--608, 2018.

\bibitem{GrossSiebert}
Mark Gross and Bernd Siebert.
\newblock From real affine geometry to complex geometry.
\newblock {\em Ann. of Math. (2)}, 174(3):1301--1428, 2011.

\bibitem{King94}
Alistair~D. King.
\newblock Moduli of representations of finite-dimensional algebras.
\newblock {\em Quart. J. Math. Oxford Ser. (2)}, 45(180):515--530, 1994.

\bibitem{KontsevichSoibelman}
Maxim Kontsevich and Yan Soibelman.
\newblock Affine structures and non-{A}rchimedean analytic spaces.
\newblock In {\em The unity of mathematics}, volume 244 of {\em Progr. Math.},
  pages 321--385. Birkh\"{a}user Boston, Boston, MA, 2006.

\bibitem{Kra91}
Henning Krause.
\newblock Maps between tree and band modules.
\newblock {\em J. Algebra}, 137(1):186--194, 1991.

\bibitem{K06}
Marina Kyureghyan.
\newblock Monotonicity checking.
\newblock In {\em General theory of information transfer and combinatorics},
  volume 4123 of {\em Lecture Notes in Comput. Sci.}, pages 735--739. Springer,
  Berlin, 2006.

\bibitem{mumford}
David Mumford, John Fogarty, and Frances Kirwan.
\newblock {\em Geometric invariant theory}, volume~34 of {\em Ergebnisse der
  Mathematik und ihrer Grenzgebiete (2) [Results in Mathematics and Related
  Areas (2)]}.
\newblock Springer-Verlag, Berlin, third edition, 1994.

\bibitem{KPY20}
Pierre-Guy Plamondon and Toshiya Yurikusa.
\newblock Tame algebras have dense $\mathbf{g}$-vector fans.
\newblock {\em arXiv preprint arXiv:2007.04215}, 2020.

\bibitem{Ru97}
Alexei Rudakov.
\newblock Stability for an abelian category.
\newblock {\em J. Algebra}, 197(1):231--245, 1997.

\bibitem{Scho91}
Aidan Schofield.
\newblock Semi-invariants of quivers.
\newblock {\em J. London Math. Soc. (2)}, 43(3):385--395, 1991.

\bibitem{STV}
Sibylle Schroll, Hipolito Treffinger, and Yadira Valdivieso.
\newblock On band modules and {$\tau$}-tilting finiteness.
\newblock {\em arXiv preprint arXiv:1911.09021}, 2019.

\bibitem{Stanley86}
Richard~P. Stanley.
\newblock Two poset polytopes.
\newblock {\em Discrete Comput. Geom.}, 1(1):9--23, 1986.

\bibitem{Treffinger2019}
Hipolito {Treffinger}.
\newblock {On sign-coherence of \(c\)-vectors}.
\newblock {\em {J. Pure Appl. Algebra}}, 223(6):2382--2400, 2019.

\bibitem{WW85}
Burkhard {Wald} and Josef {Waschb\"usch}.
\newblock {Tame biserial algebras}.
\newblock {\em {J. Algebra}}, 95:480--500, 1985.

\bibitem{W21}
Nicholas~J. Williams.
\newblock The first higher {S}tasheff--{T}amari orders are quotients of the
  higher {B}ruhat orders.
\newblock {\em arXiv preprint arXiv:2012.10371}, 2020.

\end{thebibliography}
\bibliographystyle{plain}

\end{document}